\documentclass[11pt]{amsart}
\usepackage{amsmath,amssymb,amsthm,mathtools}
\usepackage[margin=1.1in]{geometry}
\usepackage[hidelinks]{hyperref}

\newtheorem{theoremL}{Theorem}
\newenvironment{lthm}[1]{\renewcommand{\thetheoremL}{#1}\begin{theoremL}}{\end{theoremL}}
\newtheorem{theorem}{Theorem}[section]
\newtheorem{prop}[theorem]{Proposition}
\newtheorem{lemma}[theorem]{Lemma}
\newtheorem{cor}[theorem]{Corollary}
\newtheorem{conj}[theorem]{Conjecture}
\theoremstyle{definition}

\theoremstyle{remark}
\newtheorem{remark}[theorem]{Remark}
\numberwithin{equation}{section}

\newcommand{\ST}{\mu_{\mathrm{ST}}}
\newcommand{\F}{\mathbb{F}}
\newcommand{\R}{\mathbb{R}}
\newcommand{\Kl}{\mathrm{Kl}}
\newcommand{\dist}{\operatorname{dist}}
\DeclareMathOperator{\tr}{tr}
\DeclareMathOperator{\Sw}{Sw}
\DeclareMathOperator{\Sym}{Sym}
\DeclareMathOperator{\Haar}{Haar}

\title[Wasserstein rates for Kloosterman sums]{Wasserstein rates of all orders for the\\ Sato--Tate law of Kloosterman sums}
\author{Amir M. Rahimi}
\address{School of Mathematics, Institute for Research in Fundamental Sciences (IPM), P.O.~Box 19395-5746, Tehran, Iran}
\email{amrahimi@ipm.ir}

\begin{document}

\begin{abstract}
We answer, in the trace metric, a question of Kowalski and Untrau on Wasserstein distances of all orders between the distribution
$\mu_q$ of the normalized Kloosterman sums $\Kl(a;q)$, $a\in\F_q^\times$, and the Sato--Tate law $\ST$: we prove bounds for every
order $s$ and show that their exponents are optimal for Katz-type moment bounds. In the trace metric we prove
$W_s(\mu_q,\ST)\ll_s q^{-1/3}$ for $1\le s<3$, $W_3\ll q^{-1/3}(\log q)^{1/3}$, $W_s\ll_s q^{-1/4-1/(4s)}$ for $3<s<\infty$ and
$W_\infty\ll q^{-1/4}$. In the angle metric we obtain bounds of every order, including $W_2\le7q^{-1/3}$ and $W_\infty\ll q^{-1/6}$.

These estimates follow from a general theorem for probability measures satisfying Katz-type Chebyshev moment bounds. Its proof
combines an edge-sensitive discrepancy estimate with a new level-set bound for the distribution function and a quantile-displacement
argument. We construct measures showing that the resulting exponents cannot in general be improved using moment bounds of this type
alone, in the trace metric and in the angle metric for $s\le2$ and $s\ge6$ (up to logarithms); these constructions do not assert optimality for the Kloosterman measures themselves. A sharp comparison between Wasserstein
distances and the Kolmogorov distance for the Sato--Tate law exhibits the same transition at $s=3$. Finally, using the exact third
moment, we prove $W_1(\mu_q,\ST)>q^{-1/2}/10$ when $3\nmid q$.
\end{abstract}

\maketitle

\section{Introduction}

This paper answers, in the trace metric, a question of Kowalski and Untrau \cite[\S1.3, Question~(3)]{KU} on Wasserstein distances of
all orders between the distribution of Kloosterman sums and the Sato--Tate law: we prove bounds for every order and show that their
exponents are optimal for Katz-type moment bounds.

Let $\F_q$ be a finite field, $\psi$ a nontrivial additive character of $\F_q$, and
\[
\Kl(a;q)=\frac1{\sqrt q}\sum_{x\in\F_q^\times}\psi\bigl(x+ax^{-1}\bigr)\in[-2,2],\qquad
\mu_q=\frac1{q-1}\sum_{a\in\F_q^\times}\delta_{\Kl(a;q)} .
\]
By Katz's equidistribution theorem \cite[Ch.~13]{Katz}, $\mu_q$ converges to the Sato--Tate law $\ST$, of density
$\varrho(x)=\frac1{2\pi}\sqrt{4-x^2}$ on $[-2,2]$. Quantitative forms of this convergence are known: Niederreiter \cite{Nied} proved
$D(\mu_q)<10q^{-1/4}$ for the Kolmogorov distance, and Kowalski and Untrau \cite[Example~4.5]{KU} proved $W_1(\mu_q,\ST)\ll q^{-1/3}$;
their Theorem~4.4 gives the same bound for the angles $\theta_a\in[0,\pi]$, $\Kl(a;q)=2\cos\theta_a$, by Lemma~\ref{lem:angle}. They asked
\cite[\S1.3, Question~(3)]{KU} for bounds for $W_s$, $s\ne1$, better than $W_s\le\operatorname{diam}^{1-1/s}W_1^{1/s}$.

The purpose of this paper is to determine, as far as the available moment information permits, the Wasserstein rates of all orders
for these measures. The starting point is Katz's bound $|\widehat{\mu_q}(k)|\le(k+1)q^{-1/2}$, $k\ge1$, for the Chebyshev moments
(Proposition~\ref{prop:Katz}). We first prove a general theorem converting such moment bounds into Wasserstein estimates
(Theorem~\ref{thm:C}), and then apply it to Kloosterman sums. The resulting rates depend on the order $s$ and differ between the
trace and angle metrics.

We consider two metrics: the \emph{trace metric}, for probability measures on $[-2,2]$, and the \emph{angle metric}, for the
corresponding measures $\mu_q^\theta$, $\ST^\theta$ on $[0,\pi]$. By Lemma~\ref{lem:angle} the latter is also the Wasserstein metric on
central probability measures on $\mathrm{SU}_2$ with its bi-invariant geodesic distance. Since $|2\cos\theta-2\cos\theta'|\le2|\theta-\theta'|$,
angle estimates imply trace estimates. The converse comparison is more delicate near the endpoints: the Sato--Tate density vanishes
there to square-root order, and an angular displacement $\Delta\theta$ at angle $\theta$ near $0$ or $\pi$ moves the trace only by
about $\theta\,\Delta\theta$. This endpoint behaviour is responsible for the different exponents.

\begin{lthm}{A}\label{thm:A}
For every finite field $\F_q$ and every nontrivial additive character $\psi$:
\begin{enumerate}
\item In the trace metric,
\[
W_s(\mu_q,\ST)\ll_s\begin{cases}q^{-1/3},&1\le s<3,\\ q^{-1/3}(\log q)^{1/3},&s=3,\\ q^{-\frac14-\frac1{4s}},&3<s<\infty,\\ q^{-1/4},&s=\infty,\end{cases}
\]
and $W_2(\mu_q,\ST)\le12\,q^{-1/3}$.
\item In the angle metric,
\[
W_s(\mu_q^\theta,\ST^\theta)\ll_s\begin{cases}q^{-1/3},&1\le s\le2,\\ q^{-\frac5{24}-\frac1{4s}}(\log q)^{\frac{s-2}{4s}},&2\le s\le6,\\ q^{-\frac16-\frac1{2s}},&6<s<\infty,\\ q^{-1/6},&s=\infty,\end{cases}
\]
and $W_2(\mu_q^\theta,\ST^\theta)\le7\,q^{-1/3}$.
\item There are absolute constants $C,c>0$ such that at least $cq^{1/2}$ elements $a\in\F_q^\times$ satisfy
$\Kl(a;q)\ge2-Cq^{-1/3}$, and at least $cq^{1/2}$ satisfy $\Kl(a;q)\le-2+Cq^{-1/3}$.
\item If $3\nmid q$, then $W_1(\mu_q,\ST)>\frac1{10}q^{-1/2}$, $W_1(\mu_q^\theta,\ST^\theta)>\frac1{20}q^{-1/2}$ and $D(\mu_q)>\frac1{13}q^{-1/2}$.
\end{enumerate}
\end{lthm}

Part (3) improves the bound $\max_a\Kl(a;p)\ge2-cp^{-2/15}$ of \cite[remark after Prop.~4.6]{KU}; Niederreiter's bound already
gives $\gg q^{3/4}$ values with $\Kl(a;q)\ge2-Cq^{-1/6}$, and the passage to $q^{-1/3}$ is an immediate consequence of the edge
estimate of Feldheim and Sodin (Proposition~\ref{prop:local}). Parts (1)--(3) use only Katz's
moment bounds, whereas (4) uses the exact third moment: (4) is an obstruction for $\mu_q$ itself and shows that in the trace metric
$q^{-1/2}$, the rate suggested by numerical experiments (Remark~\ref{rem:numerics}), would be optimal when $3\nmid q$. The table
compares (1) and (2) with previous results. It lists upper bounds for the distance between $\mu_q$ and $\ST$; implied constants
depend on $s$, ``(log)'' stands for $(\log q)^{(s-2)/(4s)}$, and $\dagger$ marks the bounds that are optimal for Katz-type moment bounds
by Theorem~\ref{thm:D}.
\begin{center}\small\setlength{\tabcolsep}{3.5pt}
\begin{tabular}{l|ll|ll}
&\multicolumn{2}{c|}{trace metric}&\multicolumn{2}{c}{angle metric ($\mathrm{SU}_2$)}\\
&prev./obtainable&Theorem~\ref{thm:A}&prev./obtainable&Theorem~\ref{thm:A}\\ \hline
$W_1$&$q^{-1/3}$ \cite{KU}&$q^{-1/3}$ $\dagger$&$q^{-1/3}$ \cite{KU}&$q^{-1/3}$ $\dagger$\\
$W_s$, $1<s\le2$&$q^{-1/3}$ \cite{BC}&$q^{-1/3}$ $\dagger$&$q^{-1/3}$ \cite{BC}&$q^{-1/3}$ $\dagger$\\
$W_s$, $2<s<3$&$q^{-s/(6(s-1))}$ \cite{BC}&$q^{-1/3}$ $\dagger$&$q^{-s/(6(s-1))}$ \cite{BC}&$q^{-5/24-1/(4s)}$ (log)\\
$W_3$&$q^{-1/4}$ \cite{BC}&$q^{-1/3}(\log q)^{1/3}$&$q^{-1/4}$ \cite{BC}&$q^{-7/24}(\log q)^{1/12}$\\
$W_s$, $3<s\le6$&$q^{-1/6-1/(4s)}$ ($*$)&$q^{-1/4-1/(4s)}$ $\dagger$&$q^{-s/(6(s-1))}$ \cite{BC}&$q^{-5/24-1/(4s)}$ (log)\\
$W_s$, $6<s<\infty$&$q^{-1/6-1/(4s)}$ ($*$)&$q^{-1/4-1/(4s)}$ $\dagger$&$q^{-s/(6(s-1))}$ \cite{BC}&$q^{-1/6-1/(2s)}$ $\dagger$\\
$W_\infty$&$q^{-1/6}$ ($*$)&$q^{-1/4}$ $\dagger$&$q^{-1/6}$ ($**$)&$q^{-1/6}$ $\dagger$
\end{tabular}
\end{center}
Here \cite{BC} refers to the smoothing inequality \cite[Thm.~4]{BC} on compact Lie groups (Remark~\ref{rem:BC}). ($*$) is not in the literature but is obtainable from
Niederreiter's bound and Theorem~\ref{thm:B}; for $2<s<3$ it gives $q^{-1/4}$, which is weaker than \cite{BC}. ($**$) is likewise not in the literature but is obtainable as follows. Borda and Cuenin \cite[Thm.~8]{BC} give a $W_\infty$ smoothing inequality on
$S^d$, $d\ge3$, hence on $\mathrm{SU}_2$; it requires a lower bound for the mass of every ball of some radius $r$ and yields bounds of
order at least $r$. For $\mu_q$ such lower bounds hold for $r\asymp q^{-1/6}$ by Proposition~\ref{prop:local}, and in this way one
obtains $q^{-1/6}$.

Parts (1)--(3) are deduced from the following general result.

\begin{lthm}{B}\label{thm:C}
Let $\mu$ be a probability measure on $[-2,2]$ with $|\widehat\mu(k)|\le\eta(k+1)$ for all $k\ge1$, where
$\widehat\mu(k)=\int U_k\,d\mu$, $U_k(2\cos\theta)=\sin((k+1)\theta)/\sin\theta$, and $\eta>0$ (for $\eta>1$ all bounds below are trivial). Put $\ell_\eta=\log(2+1/\eta)$. Then
$D(\mu)\ll\eta^{1/2}$, with $D(\mu)\le19\eta^{1/2}$ if $\mu$ is the empirical measure of a finite multiset, and
\[
W_s(\mu,\ST)\ll_s\begin{cases}\eta^{2/3},&1\le s<3,\\ \eta^{2/3}\ell_\eta^{1/3},&s=3,\\ \eta^{\frac12+\frac1{2s}},&3<s<\infty,\\ \eta^{1/2},&s=\infty,\end{cases}
\qquad
W_s(\mu^\theta,\ST^\theta)\ll_s\begin{cases}\eta^{2/3},&1\le s\le2,\\ \eta^{\frac5{12}+\frac1{2s}}\ell_\eta^{\frac{s-2}{4s}},&2\le s\le6,\\ \eta^{\frac13+\frac1s},&6<s<\infty,\\ \eta^{1/3},&s=\infty.\end{cases}
\]
Moreover $W_2(\mu^\theta,\ST^\theta)\le(\pi^2+3.3)\eta^{2/3}$.
\end{lthm}

Such moment bounds hold for one-parameter families over open subsets of $\mathbf P^1$ with $\mathrm{SL}_2$ monodromy
(Corollary~\ref{cor:SL2}). The angle bounds for $2\le s\le6$ are obtained by interpolating between $s=2$ and $s=6$, which is possible
because the quantile coupling \eqref{eq:quantile} is optimal for every $s$ simultaneously.

\begin{lthm}{C}\label{thm:D}
Let $0<\delta\le10^{-6}$ and let $L\ge42\delta^{-2}$ be an integer. There are empirical measures $\mu_1,\mu_2,\mu_3$ of $L$-point
multisets in $[-2,2]$ with $|\widehat{\mu_i}(k)|\le(k+1)\delta$ for all $k\ge1$ and
\begin{gather*}
W_1(\mu_1,\ST)\ge\tfrac16\delta^{2/3},\qquad D(\mu_2)\ge0.7\,\delta^{1/2},\qquad W_\infty(\mu_2,\ST)\ge2.2\,\delta^{1/2},\\
W_s(\mu_2,\ST)\ge\tfrac1{100}\delta^{\frac12+\frac1{2s}}\ (s\ge2),\\ W_\infty(\mu_3^\theta,\ST^\theta)\ge1.4\,\delta^{1/3},\qquad
W_s(\mu_3^\theta,\ST^\theta)\ge\tfrac1{25}\delta^{\frac13+\frac1s}\ (s\ge1).
\end{gather*}
\end{lthm}

With $L=q-1$ and $\delta=7q^{-1/2}$ (for $q\ge5\cdot10^{13}$), this gives multisets of the same size as the support of $\mu_q$, with
moments bounded by $7(k+1)q^{-1/2}$. Thus the exponents in Theorem~\ref{thm:C} are optimal for $D$ and for every $W_s$ in the trace
metric (up to the logarithm at $s=3$), and in the angle metric for $s\le2$, $s\ge6$ (up to $\ell_\eta^{1/6}$ at $s=6$) and $s=\infty$. This concerns what moment bounds of
Katz type imply, not $\mu_q$ itself. Any improvement of Theorem~\ref{thm:A}(1) therefore needs arithmetic information on the moments;
\S\ref{sec:model} discusses what the numerical evidence suggests.

The $W_\infty$ bounds do not follow from the Kolmogorov bound. The next theorem is sharp in every regime; by its converse part,
$D(\mu_q)\ll q^{-1/4}$ alone gives only $W_\infty(\mu_q,\ST)\ll q^{-1/6}$.

\begin{lthm}{D}\label{thm:B}
Let $\mu$ be a probability measure on $[-2,2]$ with $0<D=D(\mu)\le\frac18$. Then
\[
W_s(\mu,\ST)\ll_s\begin{cases}D,&1\le s<3,\\ D(\log(1/D))^{1/3},&s=3,\\ D^{\frac23+\frac1s},&3<s<\infty,\end{cases}
\qquad W_\infty(\mu,\ST)\ll D^{2/3}.
\]
Conversely, for every $0<D\le\frac18$ there is $\mu$ with $D(\mu)=D$ for which each of these bounds is attained up to a
constant depending only on $s$.
\end{lthm}

The $\ST$-mass within distance $u$ of $\pm2$ is $\asymp u^{3/2}$, so an error $D$ in the distribution function can move quantiles
near the edges by $D^{2/3}$. For $s<3$ the bulk dominates and $W_s\asymp D$; at $s=3$ all distances from the edge contribute
equally, giving the logarithm; for $s>3$ the edge displacements dominate. For densities vanishing to order $\beta$ the threshold is
$1+1/\beta$ (Remark~\ref{rem:beta}).

\subsection*{Main contributions}
In the trace metric, Theorems~\ref{thm:A} and \ref{thm:D} answer the question of Kowalski and Untrau as far as moment bounds allow
(up to a logarithm at $s=3$). Our main result is Theorem~\ref{thm:C}, which converts Katz-type bounds for Chebyshev moments into Wasserstein
estimates of all orders. Its proof introduces two ingredients: a level-set estimate for the distribution function (Lemma~\ref{lem:level}), based on
positivity and a one-sided Lipschitz bound, and a quantile-displacement estimate (Lemma~\ref{lem:disp}) that converts local control
of the distribution function into transport bounds. Combined with the edge-sensitive Erd\H{o}s--Tur\'an estimate of Feldheim and Sodin
\cite{FS} (Proposition~\ref{prop:local}), these arguments yield
the trace-metric bounds, including the range $2<s\le3$, which is not obtained by interpolating between $W_1$ and $W_\infty$ bounds.

Theorem~\ref{thm:D} gives measures satisfying the same moment constraints for which the exponents of Theorem~\ref{thm:C} are attained in
the stated ranges; the Rudin--Shapiro construction is essential for the trace-metric lower bounds with $s\le3$. Theorem~\ref{thm:B}
gives the sharp relation between the Kolmogorov distance and Wasserstein distances for the Sato--Tate law. We also give a direct $W_2$
smoothing argument on $\mathrm{SU}_2$ with an explicit constant (Theorem~\ref{thm:W2}), and
lower bounds for the Kloosterman measures themselves from their exact third moment.

\subsection*{Related work}
Quantitative Wasserstein estimates in arithmetic settings have recently been studied by Kowalski and Untrau \cite{KU}, Borda and
Cuenin \cite{BC} and Borda and Giannoni \cite{BG}. Borda and Cuenin give general smoothing inequalities on compact spaces and Lie groups,
from which some of the rates used here can also be recovered (Remark~\ref{rem:BC}); for $W_2$ smoothing on tori see also \cite{BrS}. Our
contribution is to exploit the specific structure of the Sato--Tate law and of the Chebyshev moment bounds to obtain the full range of
trace and angle estimates. Effective Sato--Tate laws via Erd\H{o}s--Tur\'an inequalities are proved in \cite{Nied,MS,Michel,MM}, and
Erd\H{o}s--Tur\'an inequalities on compact Lie groups in \cite{Ros,FLX}. For the quantile representation of $W_s$ on the line we use
\cite{BL,Villani}.

\subsection*{Organization}
\S\ref{sec:prelim} contains preliminaries, \S\ref{sec:quantiles} proves Theorem~\ref{thm:B}, \S\ref{sec:W2} the $W_2$ inequality and
\S\ref{sec:local} the bounds of Theorem~\ref{thm:C}. Theorem~\ref{thm:D} is proved in \S\ref{sec:sharp} and Theorem~\ref{thm:A} in
\S\ref{sec:proofs}. \S\ref{sec:model} discusses a random model, and \S\ref{sec:questions} lists open questions.

\section{Preliminaries}\label{sec:prelim}

\subsection*{The quantile coupling}
For a probability measure $\lambda$ on $\R$ write $F_\lambda$ for its distribution function and $q_\lambda(t)=\inf\{x:F_\lambda(x)\ge t\}$ for
its quantile function. For compactly supported probability measures $\lambda,\nu$ on $\R$ and $1\le s<\infty$
\cite[Thm.~2.18]{Villani}, \cite{BL},
\begin{equation}\label{eq:quantile}
W_s(\lambda,\nu)^s=\int_0^1|q_\lambda-q_\nu|^s\,dt,\qquad W_\infty(\lambda,\nu)=\operatorname*{ess\,sup}_{0<t<1}|q_\lambda-q_\nu|.
\end{equation}
Moreover $W_\infty(\lambda,\nu)\le w$ as soon as $F_\lambda(x)\le F_\nu(x+w)$ and $F_\nu(x)\le F_\lambda(x+w)$ for all $x$:
if $x=q_\lambda(t)$ then $F_\nu(x+w)\ge F_\lambda(x)\ge t$, so $q_\nu(t)\le q_\lambda(t)+w$, and symmetrically.

\begin{cor}\label{cor:interp}
For compactly supported probability measures $\lambda,\nu$ on $\R$ and $1\le r\le s<\infty$,
$W_s(\lambda,\nu)\le W_\infty(\lambda,\nu)^{1-r/s}W_r(\lambda,\nu)^{r/s}$.
\end{cor}
\begin{proof}
By \eqref{eq:quantile}, $W_s^s=\int_0^1|q_\lambda-q_\nu|^{s-r}|q_\lambda-q_\nu|^r\,dt\le W_\infty^{s-r}W_r^r$.
\end{proof}

\subsection*{Chebyshev polynomials and $\mathrm{SU}_2$}

With $x=2\cos\theta$, $\varrho(x)\,dx=\frac2\pi\sin^2\theta\,d\theta$. We have $U_k(\tr g)=\chi_k(g)$, the character of $\Sym^k$ of
$\mathrm{SU}_2$; $\ST$ is the image of the Haar probability under the trace, and $\int U_jU_k\,d\ST=\delta_{jk}$. For a probability $\mu$ on $[-2,2]$ write
$\widehat\mu(k)=\int U_k\,d\mu$, $G$ for the distribution function of $\ST$, $\mathfrak q=G^{-1}$, and
$D(\mu)=\sup_x|F_\mu(x)-G(x)|$. Write $\mu^\theta$ for the image of $\mu$ under $x\mapsto\arccos(x/2)\in[0,\pi]$; thus
$\ST^\theta$ has density $\frac2\pi\sin^2\theta$ and distribution function
$\Psi(\tau)=\ST^\theta([0,\tau])=(\tau-\sin\tau\cos\tau)/\pi$. Finally, $\mu^-$ denotes the image of $\mu$ under
$x\mapsto-x$; since $U_k(-x)=(-1)^kU_k(x)$, $|\widehat{\mu^-}(k)|=|\widehat\mu(k)|$.

We identify $\mathrm{SU}_2$ with the unit quaternions, with the geodesic metric $\dist$ of $S^3$. It is bi-invariant,
an element with eigenvalues $e^{\pm i\theta}$, $0\le\theta\le\pi$, is at distance $\theta$ from $1$, and
$\tr g=2\operatorname{Re}g$, so $|\tr g-\tr g'|\le2|g-g'|\le2\dist(g,g')$. A probability on $\mathrm{SU}_2$ is
\emph{central} if it is invariant under conjugation; it is determined by the law of $\theta(g)=\dist(1,g)$.

\begin{lemma}\label{lem:angle}
For central probabilities $\nu_1,\nu_2$ on $\mathrm{SU}_2$ and $1\le s\le\infty$,
$W_s(\nu_1,\nu_2)=W_s(\theta_*\nu_1,\theta_*\nu_2)$, the right side computed on $[0,\pi]$.
\end{lemma}
\begin{proof}
Since $|\theta(g)-\theta(g')|=|\dist(1,g)-\dist(1,g')|\le \dist(g,g')$, every coupling of $\nu_1,\nu_2$ induces a coupling of
the angle laws with smaller cost, so ``$\ge$''. Conversely, let $\gamma$ be a coupling of $\theta_*\nu_1$ and
$\theta_*\nu_2$, fix a unit imaginary quaternion $i$, and let $u$ be Haar distributed and independent of
$(\theta,\theta')\sim\gamma$. Then $g=ue^{\theta i}u^{-1}$ and $g'=ue^{\theta'i}u^{-1}$ have laws $\nu_1$ and $\nu_2$, and
$\dist(g,g')=\dist(e^{\theta i},e^{\theta'i})=|\theta-\theta'|$.
\end{proof}

Niederreiter \cite{Nied} works over any finite field $\F_q$ with a nontrivial additive character $\chi$, with
$K(q,a)=\sum_{b\in\F_q^\times}\chi(b+ab^{-1})$ and $w(q,a)=K(q,a)/(2\sqrt q)$. For $\chi=\psi$ we have
$\Kl(a;q)=2w(q,a)$; his Chebyshev polynomials satisfy $U^{[-1,1]}_r(w)=U_r(2w)$, and his star discrepancy of the
$w(q,a)$ equals $D(\mu_q)$. The next two statements are therefore those of \cite{Nied}.

\begin{prop}[Katz {\cite[Ch.~13]{Katz}}; see {\cite[(6)]{Nied}}]\label{prop:Katz}
For every finite field $\F_q$ and $k\ge1$, $|\widehat{\mu_q}(k)|\le\frac{(k+1)\sqrt q}{2(q-1)}\le(k+1)q^{-1/2}$.
\end{prop}

\begin{prop}[{\cite[Lemma~3]{Nied}}]\label{prop:Nied}
Let $\mu=\frac1N\sum_{m=1}^N\delta_{x_m}$ with $x_1,\dots,x_N\in[-2,2]$. For every odd integer $k\ge1$,
\[
D(\mu)<\frac{8}{0.362\pi k+4}+\frac{4k-3}{0.362\pi k+2\pi}\sum_{r=1}^{2k-1}\frac{r+1}{r(r+2)}|\widehat\mu(r)|.
\]
\end{prop}

\section{Quantiles and the Kolmogorov distance}\label{sec:quantiles}

\begin{lemma}\label{lem:shift}
Let $\nu$ have a continuous, strictly increasing distribution function on $[c,d]$, let $\mu$ be a probability on
$[c,d]$ and $D=\sup|F_\mu-F_\nu|$. Extending $q_\nu$ by $c$ on $(-\infty,0]$ and by $d$ on $[1,\infty)$, we have
$q_\nu(t-D)\le q_\mu(t)\le q_\nu(t+D)$ for $0<t<1$.
\end{lemma}
\begin{proof}
If $F_\mu(x)\ge t$ then $F_\nu(x)\ge t-D$, so $x\ge q_\nu(t-D)$. If $t+D<1$ and $x=q_\nu(t+D)$, then
$F_\mu(x)\ge F_\nu(x)-D=t$, so $q_\mu(t)\le x$; if $t+D\ge1$ the upper bound is trivial.
\end{proof}

\begin{lemma}\label{lem:q}
We have $\mathfrak q(1-t)=-\mathfrak q(t)$, $\mathfrak q'\ge\pi$ on $(0,1)$, and for $0<t\le\frac12$
\[
c_1t^{2/3}\le \mathfrak q(t)+2\le c_2t^{2/3},\qquad c_3t^{-1/3}\le \mathfrak q'(t)\le c_4t^{-1/3}
\]
with absolute constants $c_i>0$. Consequently, with $\mathfrak q$ extended as in Lemma~\ref{lem:shift}, for $0<h\le\frac18$ and
$t_*=\min(t,1-t)$ we have $\mathfrak q(t+h)-\mathfrak q(t-h)\le Cht_*^{-1/3}$ if $t_*\ge2h$, and $\mathfrak q(t+h)-\mathfrak q(t-h)\le Ch^{2/3}$ if
$t_*<2h$, with $C$ absolute.
\end{lemma}
\begin{proof}
Symmetry is clear, and $\mathfrak q'=1/\varrho(\mathfrak q)\ge\pi$. For $0\le u\le2$, $\frac1{2\pi}\sqrt{2u}\le\varrho(-2+u)\le\frac1\pi\sqrt u$,
hence $\frac{\sqrt2}{3\pi}v^{3/2}\le G(-2+v)\le\frac2{3\pi}v^{3/2}$ for $0\le v\le2$. For $0<t\le\frac12$ and
$v=\mathfrak q(t)+2\in(0,2]$ this gives $(3\pi t/2)^{2/3}\le v\le(3\pi t/\sqrt2)^{2/3}$ and
$\pi v^{-1/2}\le \mathfrak q'(t)\le\sqrt2\pi v^{-1/2}$.

For the increments we may take $t\le\frac12$. If $t\ge2h$, then $[t-h,t+h]\subset[\frac t2,\frac58]$. For
$u\in[\frac t2,\frac12]$ we have $\mathfrak q'(u)\le c_4u^{-1/3}\le c_4(t/2)^{-1/3}$, and for $u\in[\frac12,\frac58]$ we use
$\mathfrak q'(u)=\mathfrak q'(1-u)$ with $1-u\in[\frac38,\frac12]\subset[\frac t2,\frac12]$. So $\mathfrak q(t+h)-\mathfrak q(t-h)\le2^{4/3}c_4ht^{-1/3}$.
If $t<2h$, then $\mathfrak q(t+h)-\mathfrak q(t-h)\le \mathfrak q(3h)+2\le c_2(3h)^{2/3}$.
\end{proof}

\begin{proof}[Proof of Theorem~\ref{thm:B}]
\emph{Upper bounds.} By \eqref{eq:quantile} and Lemma~\ref{lem:shift}, $|q_\mu(t)-\mathfrak q(t)|\le\Delta(t):=\mathfrak q(t+D)-\mathfrak q(t-D)$.
Since $\Delta(1-t)=\Delta(t)$, Lemma~\ref{lem:q} gives
\[
W_s^s\le2\int_0^{1/2}\Delta^s\,dt\le2C^s\Bigl(D^s\int_{2D}^{1/2}t^{-s/3}\,dt+2D^{1+2s/3}\Bigr).
\]
For $s<3$ the integral is at most $\frac3{3-s}$ and $D^{1+2s/3}\le D^s$; for $s=3$ it equals $\log\frac1{4D}$; for
$s>3$ it is at most $\frac3{s-3}(2D)^{1-s/3}$. This gives the bounds for $s<\infty$, and $W_\infty\le\sup\Delta\ll D^{2/3}$.

\emph{Lower bounds.} Let $\mu$ be the image of Lebesgue measure on $(0,1)$ under $t\mapsto \mathfrak q(\min(t+D,1))$, so that
$q_\mu(t)=\mathfrak q(\min(t+D,1))$. Then $F_\mu=\max(G-D,0)$ on $[-2,2)$ and $\mu$ has an atom of mass $D$ at $2$; hence
$0\le G-F_\mu\le D$, with equality when $D\le G(x)<1$, so $D(\mu)=D$. Since $\mathfrak q'\ge\pi$, $q_\mu(t)-\mathfrak q(t)\ge\pi D$ for
$t\le1-D$, so $W_s\ge W_1\ge\pi D(1-D)\ge2D$. For $0<t<\frac D2$, Lemma~\ref{lem:q} gives
$q_\mu(t)-\mathfrak q(t)\ge \mathfrak q(D)-\mathfrak q(\frac D2)\ge\frac D2c_3D^{-1/3}=\frac12c_3D^{2/3}$, so $W_\infty\ge\frac12c_3D^{2/3}$.
Finally, for $D\le t\le\frac14$ we have $t+D\le2t\le\frac12$, so
$q_\mu(t)-\mathfrak q(t)=\int_t^{t+D}\mathfrak q'\ge c_3D(2t)^{-1/3}$ and
\[
W_s^s\ge(c_3D)^s2^{-s/3}\int_D^{1/4}t^{-s/3}\,dt .
\]
For $s=3$ the integral is $\log\frac1{4D}\ge\frac12\log\frac1D$ when $D\le\frac1{16}$ (for $\frac1{16}<D\le\frac18$ use
$W_3\ge2D$), and for $s>3$ it is $\ge\frac3{s-3}\bigl(D^{1-s/3}-4^{s/3-1}\bigr)\gg_sD^{1-s/3}$ for $D$ small (for the remaining
$D$ use $W_s\ge2D$). This gives $W_3\gg D(\log1/D)^{1/3}$ and $W_s\gg_sD^{2/3+1/s}$.
\end{proof}

\begin{remark}\label{rem:beta}
The same argument applies to any $\nu$ on $[c,d]$ whose density is $\asymp\min(x-c,d-x)^\beta$: with $s^*=1+1/\beta$ one gets
$W_s\ll D$ for $s<s^*$, $D(\log1/D)^{1/s^*}$ for $s=s^*$, $D^{1/s+1/(\beta+1)}$ for $s>s^*$ and $W_\infty\ll D^{1/(\beta+1)}$, all attained by
$q_\mu(t)=q_\nu(\min(t+D,1))$.
\end{remark}

\section{A smoothing inequality for \texorpdfstring{$W_2$}{W2} on \texorpdfstring{$\mathrm{SU}_2$}{SU2}}\label{sec:W2}

For probability measures $\rho,\nu$ on a connected Riemannian manifold, Peyre \cite[Thm.~1]{Peyre} proved
\begin{equation}\label{eq:peyre}
W_2(\rho,\nu)\le2\|\nu-\rho\|_{\dot H^{-1}(\rho)},\qquad\|\sigma\|_{\dot H^{-1}(\rho)}=\sup\Bigl\{\Bigl|\int f\,d\sigma\Bigr|:f\in C^1,\ \int|\nabla f|^2\,d\rho\le1\Bigr\}.
\end{equation}
We apply it on $\mathrm{SU}_2=S^3$ with $\rho$ the Haar probability. The restriction of $\chi_k$ to $S^3$ is a zonal
spherical harmonic of degree $k$, so $-\Delta\chi_k=k(k+2)\chi_k$; by the spectral theorem, if $\sigma$ has density
$\sum_{k=1}^Tc_k\chi_k$ with respect to Haar measure, then
\begin{equation}\label{eq:H-1}
\|\sigma\|_{\dot H^{-1}(\Haar)}^2=\sum_{k=1}^T\frac{|c_k|^2}{k(k+2)} .
\end{equation}

The bound $W_2\ll q^{-1/3}$ also follows from \cite[Thm.~4]{BC}; the following direct argument gives small explicit
constants. We smooth by a central kernel of degree $T$ whose second moment is as small as possible; the optimal
choice comes from Gauss quadrature for the weight $\sqrt{1-x^2}$.

\begin{lemma}\label{lem:kernel}
For every integer $T\ge1$ there is a central probability $\kappa_T$ on $\mathrm{SU}_2$ with $\int\chi_k\,d\kappa_T=0$ for $k>T$ and
$\int\dist(1,h)^2\,d\kappa_T(h)\le\pi^4/T^2$.
\end{lemma}
\begin{proof}
Let $n=\lfloor T/2\rfloor$, let $x_1=\cos\frac\pi{n+2}$ be the largest zero of the Chebyshev polynomial
$V_{n+1}(\cos t)=\sin((n+2)t)/\sin t$, and $p(x)=V_{n+1}(x)/(x-x_1)$, of degree $n$. Let $\kappa_T$ be the central probability
whose density, at an element of angle $\theta$, is proportional to $J(\theta)=p(\cos\theta)^2$. Then $J$ is an even nonnegative
trigonometric polynomial of degree $2n\le T$, hence (as $2\cos j\theta=\chi_j-\chi_{j-2}$) a combination of $\chi_0,\dots,\chi_T$.
By the Weyl integration formula and $x=\cos\theta$, $\int F(\theta)\,d\kappa_T=\int_{-1}^1F\,p^2\sqrt{1-x^2}\,dx\big/\int_{-1}^1p^2\sqrt{1-x^2}\,dx$.
Gauss quadrature for the weight $\sqrt{1-x^2}$, with nodes the zeros of $V_{n+1}$ and positive weights, is exact in degree
$\le2n+1$; since $p$ vanishes at every node except $x_1$, it gives $\int\cos\theta\,d\kappa_T=x_1$. As
$\theta^2\le\frac{\pi^2}2(1-\cos\theta)$ on $[0,\pi]$,
\[
\int\dist(1,h)^2\,d\kappa_T\le\frac{\pi^2}2\Bigl(1-\cos\frac\pi{n+2}\Bigr)\le\frac{\pi^4}{4(n+2)^2}<\frac{\pi^4}{T^2}.\qedhere
\]
\end{proof}

\begin{theorem}\label{thm:W2}
For every probability $\mu$ on $[-2,2]$ and every integer $T\ge1$,
\[
W_2(\mu^\theta,\ST^\theta)\le\frac{\pi^2}T+2\Bigl(\sum_{k=1}^T\frac{|\widehat\mu(k)|^2}{k(k+2)}\Bigr)^{1/2},\qquad W_2(\mu,\ST)\le2\,W_2(\mu^\theta,\ST^\theta).
\]
\end{theorem}
\begin{proof}
The second inequality holds because $|2\cos\theta-2\cos\theta'|\le2|\theta-\theta'|$. For $x\in[-2,2]$ let $m_x$ be the invariant probability
on the conjugacy class $\{\tr g=x\}$ and $\tilde\mu=\int m_x\,d\mu(x)$, a central probability with $\theta_*\tilde\mu=\mu^\theta$ and
$\int\chi_k\,d\tilde\mu=\widehat\mu(k)$. If $g\sim\tilde\mu$ and $h\sim\kappa_T$ are independent, $(g,gh)$ couples $\tilde\mu$ and $\tilde\mu*\kappa_T$
with $\dist(g,gh)=\dist(1,h)$, so $W_2(\tilde\mu,\tilde\mu*\kappa_T)\le\pi^2/T$. Let $\rho_k$ be the irreducible representation with
character $\chi_k$. For central $\nu$, Schur's lemma gives $\int\rho_k\,d\nu=(k+1)^{-1}(\int\chi_k\,d\nu)I$, and
$\int\rho_k\,d(\nu_1*\nu_2)=\int\rho_k\,d\nu_1\int\rho_k\,d\nu_2$. Hence $\tilde\mu*\kappa_T$ has density $1+\sum_{k\le T}c_k\chi_k$ with
$c_k=\widehat\mu(k)b_k/(k+1)$, $b_k=\int\chi_k\,d\kappa_T$, $|b_k|\le k+1$, so $|c_k|\le|\widehat\mu(k)|$. By \eqref{eq:peyre} and \eqref{eq:H-1},
$W_2(\tilde\mu*\kappa_T,\Haar)\le2(\sum_{k\le T}|\widehat\mu(k)|^2/(k(k+2)))^{1/2}$. The triangle inequality and Lemma~\ref{lem:angle} complete
the proof.
\end{proof}

\begin{cor}\label{cor:W2Kl}
For every finite field, $W_2(\mu_q^\theta,\ST^\theta)\le7q^{-1/3}$ and $W_2(\mu_q,\ST)\le12q^{-1/3}$.
\end{cor}
\begin{proof}
By Proposition~\ref{prop:Katz} and $(k+1)^2/(k(k+2))=1+1/(k(k+2))$, the sum in Theorem~\ref{thm:W2} is at most
$(T+\frac34)q/(4(q-1)^2)$. Take $T=\lceil4q^{1/3}\rceil$. Then $\pi^2/T\le\frac{\pi^2}4q^{-1/3}<2.47q^{-1/3}$ and
$\frac{\sqrt q}{q-1}(T+\frac34)^{1/2}\le q^{-1/3}\cdot\frac q{q-1}(4+1.75q^{-1/3})^{1/2}$. The last two factors decrease in $q$,
and their product is $<3.43$ at $q=3$. Hence $W_2(\mu_q^\theta,\ST^\theta)<5.9q^{-1/3}$ for $q\ge3$. For $q=2$ use the trivial bounds
$W_2(\mu_q^\theta,\ST^\theta)\le\pi<7\cdot2^{-1/3}$ and $W_2(\mu_q,\ST)\le4<12\cdot2^{-1/3}$.
\end{proof}

\section{Local discrepancy and the bounds of Theorem~\ref{thm:C}}\label{sec:local}

We use the cap estimate of Proposition~\ref{prop:local} (which gives the $W_\infty$ bounds), the level-set estimate of
Lemma~\ref{lem:level} and the displacement estimate of Lemma~\ref{lem:disp}; Theorem~\ref{thm:sharp} combines them over dyadic regions
at distance $\asymp2^{-j}$ from the edges.

Fix once and for all a nonincreasing $\varpi\in C^\infty(\R)$ with $\varpi=1$ on $(-\infty,0]$ and $\varpi=0$ on $[1,\infty)$.
Constants $C$ below may depend on $\varpi$ and are otherwise absolute.

\begin{lemma}\label{lem:cap}
Let $0<\varepsilon\le\tau\le\frac\pi2$, $H(\theta)=\varpi((\theta-\tau)/\varepsilon)$ for $\theta\in[0,\pi]$, and
$h(x)=H(\arccos(x/2))$. Then $h=\sum_{k\ge0}\widehat h_kU_k$ uniformly on $[-2,2]$, with $\widehat h_k=\int hU_k\,d\ST$ and
\[
\sum_{k\ge1}(k+1)|\widehat h_k|\le C\Bigl(1+\frac\tau\varepsilon\Bigr).
\]
\end{lemma}
\begin{proof}
Since $H$ is constant near $0$ and near $\pi$, $\theta\mapsto H(\theta)\sin\theta$ extends to a smooth odd $2\pi$-periodic
function, whose sine coefficients are
$\widehat h_k=\frac2\pi\int_0^\pi H(\theta)\sin\theta\sin((k+1)\theta)\,d\theta$; they decay rapidly, which gives the uniform
expansion. Let $g_k(\theta)=\frac{\sin k\theta}k-\frac{\sin(k+2)\theta}{k+2}$, so that $g_k(0)=0$ and
$g_k'(\theta)=2\sin\theta\sin((k+1)\theta)$. Integrating by parts, with $w=-H'\ge0$, $\int w=1$ and
$\operatorname{supp}w\subset[\tau,\tau+\varepsilon]\subset[\tau,2\tau]$,
\[
\widehat h_k=\frac1\pi\int w(\theta)g_k(\theta)\,d\theta .
\]
We use three bounds. (i) $|g_k(\theta)|\le2\int_0^\theta t\cdot(k+1)t\,dt=\frac23(k+1)\theta^3$, so
$|\widehat h_k|\le\frac{16}{3\pi}(k+1)\tau^3$. (ii) Since
\begin{equation}\label{eq:gk}
g_k(\theta)=\frac{2\sin k\theta}{k(k+2)}-\frac{2\sin\theta\cos((k+1)\theta)}{k+2},
\end{equation}
$|\widehat h_k|\le\frac2\pi\bigl(\frac1{k(k+2)}+\frac{2\tau}{k+2}\bigr)$. (iii) Integrating by parts twice in each term of
\eqref{eq:gk}, and using $\|w''\|_1=\|\varpi'''\|_1\varepsilon^{-2}$ and
$\|(w\sin)''\|_1\le2\tau\|\varpi'''\|_1\varepsilon^{-2}+2\|\varpi''\|_1\varepsilon^{-1}+1\le C\tau\varepsilon^{-2}$ (as $\varepsilon\le\tau\le\frac\pi2$),
we get $|\widehat h_k|\le C(k^{-4}+\tau k^{-3})\varepsilon^{-2}$.

Use (i) for $k+1\le1/\tau$, (ii) for $1/\tau<k+1\le1/\varepsilon$, and (iii) for $k+1>1/\varepsilon$. The three ranges
contribute at most $\frac{16}{3\pi}\tau^3\sum_{m\le1/\tau}m^2\le C$, then
$\frac2\pi\sum_{1/\tau-1<k\le1/\varepsilon}(\frac1k+2\tau)\le C(1+\log\frac{2\tau}\varepsilon+\frac\tau\varepsilon)$, and finally
$C\varepsilon^{-2}\sum_{k+1>1/\varepsilon}(k^{-3}+\tau k^{-2})\le C(1+\frac\tau\varepsilon)$. Since $\log z\le z$, the lemma follows.
\end{proof}

The following edge-sensitive Erd\H{o}s--Tur\'an estimate, which refines the uniform Chebyshev version of Grabner
\cite{Grabner}, is due to Feldheim and Sodin: it follows from \cite[Prop.~5]{FS}
(applied to the image of $\mu$ under $x\mapsto x/2$, with $\rho=\max(1-|x_0|,n_0^{-2})$ and $n_0$ optimized). We include the short
proof, which is needed with the specific test functions of Lemma~\ref{lem:cap}, for completeness. The $W_\infty$ bounds below and
Theorem~\ref{thm:A}(3) follow from it quickly.

\begin{prop}[local discrepancy {\cite{FS}}]\label{prop:local}
Let $\mu$ satisfy the hypothesis of Theorem~\ref{thm:C} with $\eta\le1$. Then for $0\le\tau\le\frac\pi2$,
\[
\bigl|\mu^\theta(I)-\Psi(\tau)\bigr|\le C_1\bigl(\eta+\eta^{1/2}\tau^{3/2}\bigr)\qquad\text{for }I=[0,\tau]\text{ and }I=[0,\tau).
\]
The same holds for $\mu^-$, i.e.\ for the caps $[\pi-\tau,\pi]$ and $(\pi-\tau,\pi]$.
\end{prop}
\begin{proof}
For $h$ as in Lemma~\ref{lem:cap} (with any cap radius $\tau'\ge\varepsilon$), uniform convergence and
$\int U_k\,d\ST=0$ for $k\ge1$ give
\begin{equation}\label{eq:hdiff}
\Bigl|\int h\,d\mu-\int h\,d\ST\Bigr|=\Bigl|\sum_{k\ge1}\widehat h_k\widehat\mu(k)\Bigr|\le\eta\sum_{k\ge1}(k+1)|\widehat h_k|\le C\eta\Bigl(1+\frac{\tau'}\varepsilon\Bigr).
\end{equation}
Let $0<2\varepsilon\le\tau$. The function $H_+(\theta)=\varpi((\theta-\tau)/\varepsilon)$ majorizes $\mathbf 1_{[0,\tau]}$, and
$H_-(\theta)=\varpi((\theta-\tau+\varepsilon)/\varepsilon)$, with cap radius $\tau-\varepsilon\ge\varepsilon$, minorizes $\mathbf 1_{[0,\tau)}$. By
\eqref{eq:hdiff},
\[
\Psi(\tau-\varepsilon)-C\eta(1+\tfrac\tau\varepsilon)\le\mu^\theta([0,\tau))\le\mu^\theta([0,\tau])\le\Psi(\tau+\varepsilon)+C\eta(1+\tfrac\tau\varepsilon),
\]
and $\Psi(\tau+\varepsilon)-\Psi(\tau-\varepsilon)\le\frac2\pi\cdot2\varepsilon(\tau+\varepsilon)^2\le\frac9\pi\varepsilon\tau^2$. If $\tau^3\ge4\eta$, take
$\varepsilon=(\eta/\tau)^{1/2}$, so that $2\varepsilon\le\tau$, $\eta\tau/\varepsilon=\eta^{1/2}\tau^{3/2}$ and $\varepsilon\tau^2=\eta^{1/2}\tau^{3/2}$. If
$\tau^3<4\eta$, let $\tau_1=(4\eta)^{1/3}$; if $\tau_1\le\frac\pi2$, the previous case and monotonicity give
$\mu^\theta([0,\tau])\le\mu^\theta([0,\tau_1])\le\Psi(\tau_1)+C\eta\le C\eta$, and $\Psi(\tau)\le\Psi(\tau_1)\le C\eta$; if
$\tau_1>\frac\pi2$ the bound is trivial. The statement for $\mu^-$ follows since $\mu^-$ satisfies the same hypothesis.
\end{proof}

For $\tau\in[0,\pi]$ let $m(\tau)=\min(\tau,\pi-\tau)$. Writing $\mu^\theta([0,\tau])=1-\mu^\theta((\tau,\pi])$ for
$\tau>\frac\pi2$ and using that $\Psi(\tau)=1-\Psi(\pi-\tau)$, Proposition~\ref{prop:local} gives, for every $\tau\in[0,\pi]$,
\begin{equation}\label{eq:E}
\bigl|\mu^\theta([0,\tau])-\Psi(\tau)\bigr|,\ \bigl|\mu^\theta([0,\tau))-\Psi(\tau)\bigr|\le E(\tau):=C_1\bigl(\eta+\eta^{1/2}m(\tau)^{3/2}\bigr).
\end{equation}

\begin{theorem}\label{thm:Winf}
Under the hypothesis of Theorem~\ref{thm:C},
\[
D(\mu)\le C\eta^{1/2},\qquad W_\infty(\mu,\ST)\le C\eta^{1/2},\qquad W_\infty(\mu^\theta,\ST^\theta)\le C\eta^{1/3}.
\]
\end{theorem}
\begin{proof}
If $\eta>1$ the bounds are trivial ($D\le1$, and the diameters are $4$ and $\pi$), so let $\eta\le1$. Since
$F_\mu(2\cos\tau)=1-\mu^\theta([0,\tau))$ and $G(2\cos\tau)=1-\Psi(\tau)$, \eqref{eq:E} gives $D(\mu)\le3C_1\eta^{1/2}$.

\emph{Angle metric.} Let $\Phi(\tau)=\mu^\theta([0,\tau])$ and $w=C_2\eta^{1/3}$ with $C_2\ge1$ to be chosen; since $W_\infty\le\pi$, we may assume $w\le\frac\pi2$. By the criterion
after \eqref{eq:quantile} it suffices to show $\Phi(\tau)\le\Psi(\tau+w)$ and $\Psi(\tau)\le\Phi(\tau+w)$ for all $\tau$, where
both are trivial if $\tau+w\ge\pi$. By \eqref{eq:E} it suffices that
\begin{equation}\label{eq:massang}
\Psi(\tau+w)-\Psi(\tau)\ge\max\bigl(E(\tau),E(\tau+w)\bigr)\qquad(0\le\tau\le\pi-w).
\end{equation}
Put $M=m(\tau)+w\ge w$. Since $m$ is $1$-Lipschitz, the right side is at most $C_1(\eta+\eta^{1/2}M^{3/2})$. Since
$\sin t\ge\frac2\pi m(t)$, the left side is at least $\frac8{\pi^3}\int_\tau^{\tau+w}m(t)^2\,dt\ge c_0wM^2$ with an absolute
$c_0>0$ (if $m(\tau)\le w$: as $m$ has slope $\pm1$, its values on $[\tau,\tau+w]$ cover an interval of length $\ge\frac w2$, so the integral is $\ge\int_0^{w/2}t^2\,dt\gg w^3\gg wM^2$; otherwise $m(t)\ge\frac12m(\tau)$ on
an interval of length $\ge\frac w2$ inside $[\tau,\tau+w]$, so the integral is $\gg wm(\tau)^2\gg wM^2$). Now
$\eta=\eta^{1/3}\eta^{2/3}\le\eta^{1/3}M^2C_2^{-2}$ and $\eta^{1/2}M^{3/2}=\eta^{1/3}M^2\cdot\eta^{1/6}M^{-1/2}\le\eta^{1/3}M^2C_2^{-1/2}$, so
\eqref{eq:massang} holds once $c_0C_2\ge C_1(C_2^{-2}+C_2^{-1/2})$.

\emph{Trace metric.} Let $w=C_2'\eta^{1/2}$; since $W_\infty\le4$, we may assume $w\le2$. As before it suffices that
$\ST([x,x+w])\ge E(\tau(y))$ for $y\in\{x,x+w\}\subset[-2,2]$, where $\tau(y)=\arccos(y/2)$. Let $B$ be the maximum of
$u(y)=\min(2-y,2+y)$ on $[x,x+w]$; then $B\ge w$, $u\le B$ on $[x,x+w]$, and since $u$ is $1$-Lipschitz, $u\ge B-\frac w2\ge\frac B2$
on a subinterval of length $\frac w2$. As $\varrho(y)\ge\frac1{\sqrt2\pi}u(y)^{1/2}$, we get $\ST([x,x+w])\ge c_0'wB^{1/2}$. Since
$2-y=4\sin^2(\tau(y)/2)\ge\frac4{\pi^2}\tau(y)^2$ and similarly at $-2$, $m(\tau(y))\le\frac\pi2\min(2-y,2+y)^{1/2}$, so
$E(\tau(y))\le C(\eta+\eta^{1/2}B^{3/4})$. As $B\le4$ and $\eta^{1/2}=w/C_2'\le B/C_2'\le2B^{1/2}/C_2'$, this is at most
$C\eta^{1/2}B^{1/2}(2/C_2'+\sqrt2)$, which is $\le c_0'C_2'\eta^{1/2}B^{1/2}$ for $C_2'$ large.
\end{proof}

\subsection{The distribution function and its level sets}
Under the hypothesis of Theorem~\ref{thm:C}, let $f(\theta)=\mu^\theta([0,\theta])-\Psi(\theta)$ for $0<\theta<\pi$. Then $|f|\le E$ by
\eqref{eq:E}, $f(0^+)=\mu^\theta(\{0\})\ge0$ and $f(\pi^-)=-\mu^\theta(\{\pi\})\le0$. Moreover, since $\theta\mapsto\mu^\theta([0,\theta])$ is
nondecreasing and $\Psi'=\frac2\pi\sin^2$,
\begin{equation}\label{eq:onesided}
f(\theta_2)-f(\theta_1)\ge-\frac2\pi\int_{\theta_1}^{\theta_2}\sin^2\qquad(\theta_1<\theta_2).
\end{equation}

\begin{lemma}\label{lem:fourier}
For $m\ge1$, $\frac2\pi\int_0^\pi f(\theta)\sin m\theta\,d\theta=\frac{\widehat\mu(m)-\widehat\mu(m-2)}{\pi m}$, where $\widehat\mu(0)$ and $\widehat\mu(-1)$ are
replaced by $0$. In particular these coefficients are at most $\frac{2\eta}\pi$ in absolute value.
\end{lemma}
\begin{proof}
Let $\sigma=\mu^\theta-\ST^\theta$. By Fubini and $\sigma([0,\pi])=0$,
$\int_0^\pi f(\theta)\sin m\theta\,d\theta=\int\frac{\cos m\theta'-\cos m\pi}m\,d\sigma(\theta')=\frac1m\int\cos m\theta\,d\sigma$, and
$2\cos m\theta=U_m(2\cos\theta)-U_{m-2}(2\cos\theta)$ for $m\ge2$, $2\cos\theta=U_1(2\cos\theta)$. Finally
$|\widehat\mu(m)|+|\widehat\mu(m-2)|\le2m\eta$ for $m\ge3$, and $|\widehat\mu(1)|\le2\eta$, $|\widehat\mu(2)|\le3\eta$.
\end{proof}

\begin{lemma}[level sets]\label{lem:level}
Let $\lambda>0$, let $I\subset(0,\pi)$ be an interval and $\varkappa>0$, and suppose that $b=\pi\lambda/(4\varkappa)\le\pi$ and that
$\sin^2\le\varkappa$ on the $b$-neighbourhood of $I$ in $(0,\pi)$. Then $|\{\theta\in I:|f(\theta)|\ge\lambda\}|\le82\varkappa\eta^2/\lambda^3$.
\end{lemma}
\begin{proof}
Let $F$ be the odd $2\pi$-periodic extension of $f$, and $A^\pm F(\theta)=b^{-1}\int_0^b F(\theta\pm u)\,du$. Applied to
$\sin m\theta$, $A^\pm$ produces a trigonometric polynomial of frequency $m$ and amplitude $|\sin(mb/2)|/(mb/2)$. By
Lemma~\ref{lem:fourier} and Parseval, using $\min(1,a)\le2a/(1+a)$ and comparison with an integral,
\[
\int_{-\pi}^\pi|A^\pm F|^2\le\pi\Bigl(\frac{2\eta}\pi\Bigr)^2\sum_{m\ge1}\min\Bigl(1,\frac4{m^2b^2}\Bigr)\le\frac{4\eta^2}\pi\int_0^\infty\frac{8\,dx}{x^2b^2+4}=\frac{8\eta^2}b .
\]
If $\theta_0\in I$ and $f(\theta_0)\ge\lambda$, then by \eqref{eq:onesided} $f\ge\lambda-\frac{2\varkappa}\pi(\theta-\theta_0)\ge\frac\lambda2$ for
$\theta\in[\theta_0,\theta_0+b]$, as long as $\theta<\pi$; since $f(\pi^-)\le0$, this interval lies in $(0,\pi)$, so $A^+F(\theta_0)\ge\frac\lambda2$.
Likewise, if $f(\theta_0)\le-\lambda$, then $f\le-\frac\lambda2$ on $[\theta_0-b,\theta_0]$, which lies in $(0,\pi)$ because $f(0^+)\ge0$, so
$A^-F(\theta_0)\le-\frac\lambda2$. By Chebyshev's inequality the set in question has measure at most
$2\cdot\frac4{\lambda^2}\cdot\frac{8\eta^2}b=\frac{256\varkappa\eta^2}{\pi\lambda^3}$.
\end{proof}

\begin{lemma}[displacement]\label{lem:disp}
There is an absolute $c_5>0$ with the following property. Let $0\le\alpha<\alpha+w\le\pi$ with $w\le\frac\pi2$; then
$\ST^\theta([\alpha,\alpha+w])\ge c_5w(m(\alpha)+w)^2$ and $\ST^\theta([\alpha,\alpha+w])\ge c_5w(m(\alpha+w)+w)^2$. Consequently, let
$t\in(0,1)$, $\tau=q_{\ST^\theta}(t)$, $\Delta=q_{\mu^\theta}(t)-\tau$, and suppose $0<|\Delta|\le R\le\frac\pi2$. Then, with
$M_Rf(\tau)=\sup_{0<r\le R}(2r)^{-1}\int_{\tau-r}^{\tau+r}|f|$ ($f=0$ outside $(0,\pi)$) and $Q=32c_5^{-1}M_Rf(\tau)$,
\[
|\Delta|\le\min\Bigl(\frac Q{m(\tau)^2},Q^{1/3}\Bigr),\qquad|2\cos(\tau+\Delta)-2\cos\tau|\le2\min\Bigl(\frac Q{m(\tau)},Q^{2/3}\Bigr).
\]
\end{lemma}
\begin{proof}
The mass bounds were shown in the proof of Theorem~\ref{thm:Winf} (the second follows from the first by the symmetry
$\theta\mapsto\pi-\theta$). If $\Delta>0$, then $\mu^\theta([0,\theta])<t=\Psi(\tau)$ for $\theta<\tau+\Delta$, so
$|f(\theta)|\ge\ST^\theta([\tau,\theta])\ge\ST^\theta([\tau,\tau+\frac\Delta2])\ge\frac{c_5}8\Delta(m(\tau)+\Delta)^2$ for
$\theta\in[\tau+\frac\Delta2,\tau+\Delta)$. If $\Delta<0$, then $\mu^\theta([0,\theta])\ge t$ for $\theta\ge\tau+\Delta$, so
$f(\theta)\ge\ST^\theta((\theta,\tau])\ge\ST^\theta([\tau-\frac{|\Delta|}2,\tau])\ge\frac{c_5}8|\Delta|(m(\tau)+|\Delta|)^2$ for
$\theta\in[\tau+\Delta,\tau+\frac\Delta2]$. In both cases $|f|$ is that large on an interval of length $\frac{|\Delta|}2$ inside
$[\tau-|\Delta|,\tau+|\Delta|]$, so $M_Rf(\tau)\ge\frac{c_5}{32}|\Delta|(m(\tau)+|\Delta|)^2$, i.e.\ $|\Delta|(m(\tau)+|\Delta|)^2\le Q$. This gives
the bound for $|\Delta|$. Next, $|2\cos(\tau+\Delta)-2\cos\tau|\le2|\Delta|\max\sin\le2P$ with $P=|\Delta|(m(\tau)+|\Delta|)$, the maximum
over the segment, since $\sin\le m$ there. As $m(\tau)+|\Delta|\ge\max(m(\tau),P^{1/2})$, we have $P\le Q/\max(m(\tau),P^{1/2})$, so
$P\le\min(Q/m(\tau),Q^{2/3})$.
\end{proof}

\begin{theorem}\label{thm:sharp}
Under the hypothesis of Theorem~\ref{thm:C}, the bounds for $W_s(\mu,\ST)$, $1\le s<\infty$, stated there hold, and
$W_6(\mu^\theta,\ST^\theta)\ll\eta^{1/2}\ell_\eta^{1/6}$, $W_s(\mu^\theta,\ST^\theta)\ll_s\eta^{1/3+1/s}$ for $s>6$.
\end{theorem}
\noindent In the bulk, $\int|f|^s\ll\eta^{2s/3}$ for $s<3$, while for $s>3$ the values of $f$ of size $\eta^{1/2}$ dominate and
give $\eta^{(s+1)/2}$; at $s=3$ all scales contribute equally. In the angle metric with $s\ge6$ the cap $\theta\lesssim\eta^{1/3}$ dominates.

\begin{proof}
We may assume $\eta\le\eta_0$ for a small absolute $\eta_0$, since otherwise all bounds are trivial; and $s>1$, since
$W_1\le W_{3/2}$. Since $x\mapsto\arccos(x/2)$ is decreasing, the quantile couplings correspond, and by \eqref{eq:quantile} and the
substitution $t=\Psi(\tau)$,
\[
W_s(\mu,\ST)^s=\int_0^\pi\bigl|2\cos q_{\mu^\theta}(\Psi(\tau))-2\cos\tau\bigr|^s\,d\Psi(\tau),\qquad
W_s(\mu^\theta,\ST^\theta)^s=\int_0^\pi|\Delta(\tau)|^s\,d\Psi(\tau),
\]
with $\Delta(\tau)=q_{\mu^\theta}(\Psi(\tau))-\tau$. By Theorem~\ref{thm:Winf}, $|\Delta|\le R:=C_2\eta^{1/3}$ almost everywhere. Applying
everything to $\mu^-$ as well, it suffices to bound the integrals over $(0,\frac\pi2]$, where $m(\tau)=\tau$. Let
$\tau_*=C_7\eta^{1/3}$ with $C_7\ge4C_2$ large.

\emph{The edge $\tau\le\tau_*$.} Here $M_Rf(\tau)\le\sup_{\theta\le\tau_*+R}E(\theta)\le C\eta$, so by Lemma~\ref{lem:disp} the integrands are
$\le C\eta^{2s/3}$ and $\le C\eta^{s/3}$, and $\Psi(\tau_*)\le C\eta$. The contributions are $\le C\eta^{2s/3+1}$ and $\le C\eta^{s/3+1}$.

\emph{Dyadic regions.} For $\rho_j=\frac\pi2 2^{-j}>\tau_*$ let $I_j=(\rho_j/2,\rho_j]$ and $I_j^*=(\rho_j/4,5\rho_j/4]$. For $\tau\in I_j$,
$M_Rf(\tau)=M_R(f\mathbf 1_{I_j^*})(\tau)$ since $R\le\tau_*/4<\rho_j/4$, and $d\Psi\le\frac2\pi\rho_j^2\,d\tau$. By Lemma~\ref{lem:disp} and
the Hardy--Littlewood maximal theorem ($s>1$), the contributions of $I_j$ are at most
\[
C_s\rho_j^{2-s}\mathcal N_j\quad\text{(trace)},\qquad C_s\rho_j^{2-2s}\mathcal N_j\quad\text{(angle)},\qquad \mathcal N_j=\int_{I_j^*}|f|^s .
\]
On $I_j^*$, since $m(\theta)\le\theta$, $|f|\le\lambda_j:=C_1(\eta+\eta^{1/2}(5\rho_j/4)^{3/2})\le C_6\eta^{1/2}\rho_j^{3/2}$ (as $\rho_j\ge\eta^{1/3}$). Apply Lemma~\ref{lem:level} with $I=I_j^*$ and
$\varkappa=4\rho_j^2$: for $\lambda\le\lambda_j$, $b=\pi\lambda/(16\rho_j^2)\le C\eta^{1/2}\rho_j^{-1/2}\le\rho_j/4$ if $C_7$ is large, so the
$b$-neighbourhood of $I_j^*$ lies in $(0,\frac32\rho_j]$, where $\sin^2\le\varkappa$. Hence
$|\{I_j^*:|f|\ge\lambda\}|\le\min(\rho_j,328\rho_j^2\eta^2\lambda^{-3})$ for $\lambda\le\lambda_j$, and $\mathcal N_j=s\int_0^{\lambda_j}\lambda^{s-1}|\{|f|\ge\lambda\}|\,d\lambda$.
With $\lambda_0=(328\rho_j\eta^2)^{1/3}$ this gives
\begin{align*}
&\text{(a) }\mathcal N_j\le\rho_j\lambda_j^s\le C_s\eta^{s/2}\rho_j^{1+3s/2};&&\text{(b) }\mathcal N_j\le C_s\rho_j\lambda_0^s\le C_s\eta^{2s/3}\rho_j^{1+s/3}\ (s<3);\\
&\text{(c) }\mathcal N_j\le C\rho_j^2\eta^2\ell_\eta;&&\text{(d) }\mathcal N_j\le C_s\rho_j^2\eta^2\lambda_j^{s-3}\le C_s\eta^{\frac{s+1}2}\rho_j^{\frac{3s-5}2}\ (s>3).
\end{align*}

\emph{Trace metric.} By (b), (c), (d), the contribution of $I_j$ is $\le C_s\eta^{2s/3}\rho_j^{3-2s/3}$ for $s<3$, $\le C\eta^2\ell_\eta\rho_j$ for
$s=3$, and $\le C_s\eta^{(s+1)/2}\rho_j^{(s-1)/2}$ for $s>3$. The exponents of $\rho_j$ are positive, so the sums over $j$ are
$\ll_s\eta^{2s/3}$, $\eta^2\ell_\eta$ and $\eta^{(s+1)/2}$, which dominate the edge contribution. This gives the trace bounds.

\emph{Angle metric, $s\ge6$.} Let $\rho_*=\eta^{1/7}$. For $\rho_j\le\rho_*$ use (a): the contribution is $\le C_s\eta^{s/2}\rho_j^{3-s/2}$,
whose sum is $\ll\eta^3\ell_\eta$ if $s=6$ and $\ll_s\eta^{s/3+1}$ if $s>6$ (dominated by $\rho_j\approx\tau_*$). For $\rho_j>\rho_*$ use
(d): the contribution is $\le C_s\eta^{(s+1)/2}\rho_j^{-(s+1)/2}$, whose sum is $\ll_s\eta^{3(s+1)/7}\le\eta^{s/3+1}$, as $s\ge6$. With the
edge term $C\eta^{s/3+1}$ this gives $W_s(\mu^\theta,\ST^\theta)^s\ll_s\eta^{s/3+1}$ for $s>6$ and $\ll\eta^3\ell_\eta$ for $s=6$.
\end{proof}

\begin{proof}[Proof of the angle bounds in Theorem~\ref{thm:C}]
For $s\le2$, $W_s\le W_2\le(\pi^2+3.3)\eta^{2/3}$ by Theorem~\ref{thm:W2}; for $s\ge6$ use Theorem~\ref{thm:sharp}, and for $s=\infty$
Theorem~\ref{thm:Winf}. For $2<s<6$ write $\frac1s=\frac{1-t}2+\frac t6$, i.e.\ $t=\frac32-\frac3s$. By \eqref{eq:quantile} and H\"older's
inequality (log-convexity of $L^p$ norms) applied to $\Delta=q_{\mu^\theta}-q_{\ST^\theta}$ on $(0,1)$,
\[
W_s(\mu^\theta,\ST^\theta)\le W_2^{1-t}W_6^t\ll\eta^{\frac23(1-t)}\bigl(\eta^{1/2}\ell_\eta^{1/6}\bigr)^t=\eta^{\frac5{12}+\frac1{2s}}\ell_\eta^{\frac{s-2}{4s}} .\qedhere
\]
\end{proof}

\section{Optimality: proof of Theorem~\ref{thm:D}}\label{sec:sharp}

We construct absolutely continuous measures $\nu$ and then discretize them (Lemma~\ref{lem:discr}). The three constructions
correspond to the three regimes of Theorem~\ref{thm:C}:
Rudin--Shapiro polynomials give many small oscillations of the density with maximal moments up to $k\approx\delta^{-2/3}$ (optimal
for $s\le3$); a single bump in the bulk moves mass $\asymp\delta^{1/2}$ by $\asymp\delta^{1/2}$ (optimal for $D$, for $W_s$ with $s>3$
and for $W_\infty$ in the trace metric); and a hole at the edge removes the mass of a cap of angle $\asymp\delta^{1/3}$ (optimal for
$W_\infty$ and $s>6$ in the angle metric).

\subsection{Discretization}
\begin{lemma}\label{lem:discr}
Let $\nu$ be a probability on $[-2,2]$ with $\nu\le2\ST$, let $0<\delta\le\frac1{10}$, and suppose $|\widehat\nu(k)|\le\frac9{10}(k+1)\delta$ for
$k\ge1$. Let $L\ge42\delta^{-2}$ be an integer, $x_i=q_\nu((i-\frac12)/L)$ for $1\le i\le L$, and $\mu=\frac1L\sum_i\delta_{x_i}$. Then
$|\widehat\mu(k)|\le(k+1)\delta$ for all $k\ge1$, $\sup|F_\mu-F_\nu|\le\frac1{2L}$ and $W_1(\mu,\nu)\le\frac2L$. The $x_i$ lie in the closed
support of $\nu$, and $W_\infty(\mu,\nu)$ is at most the largest length of the intervals $[q_\nu(\frac{i-1}L),q_\nu(\frac iL)]$, where
$q_\nu(0)=\min\operatorname{supp}\nu$.
\end{lemma}
\begin{proof}
$F_\nu(x_i)=(i-\frac12)/L$ and $F_\mu$ jumps from $\frac{i-1}L$ to $\frac iL$ at $x_i$, so $\sup|F_\mu-F_\nu|\le\frac1{2L}$ and
$W_1(\mu,\nu)=\int|F_\mu-F_\nu|\le\frac2L$. For $t\in(\frac{i-1}L,\frac iL]$, $q_\mu(t)=x_i$ and $q_\nu(t)$ lie in the $i$-th interval, which
gives the claim on $W_\infty$.

\emph{Small $k$.} Integrating by parts, $|\int f\,d(\mu-\nu)|\le V(f)\sup|F_\mu-F_\nu|$ for $f\in C^1$. Since
$U_k(2\cos\theta)=\sum_{|j|\le k,\,j\equiv k\,(2)}e^{ij\theta}$, $|\frac d{d\theta}U_k(2\cos\theta)|\le\frac{(k+1)^2}2$ and $V(U_k)\le\frac\pi2(k+1)^2$.
Let $k_0=\lceil5.1/\delta\rceil$, so $k_0+1\le5.3/\delta$. For $k\le k_0$,
\[
|\widehat\mu(k)|\le\tfrac9{10}(k+1)\delta+\frac{\pi(k+1)^2}{4L}\le(k+1)\Bigl(\tfrac9{10}\delta+\frac{5.3\pi}{168}\delta\Bigr)\le(k+1)\delta .
\]
\emph{Large $k$.} Let $g(x)=2/\sqrt{4-x^2}$, so $|U_k|\le g$. As $g$ decreases on $[-2,0]$ and increases on $[0,2]$, for each $i$ there is
a half-cell $J_i\subset[\frac{i-1}L,\frac iL]$ of length $\frac1{2L}$ on which $g(q_\nu(t))\ge g(x_i)$; the $J_i$ are disjoint, so for $k>k_0$
\[
|\widehat\mu(k)|\le\frac1L\sum_ig(x_i)\le2\int_0^1g(q_\nu(t))\,dt=2\int g\,d\nu\le4\int g\,d\ST=\frac{16}\pi<5.1\le k_0\delta<(k+1)\delta.\qedhere
\]
\end{proof}

\subsection{Perturbations of $\ST$}
For a bounded function $r$ on $[0,\pi]$ with $r\ge-1$ and $\int_0^\pi r\sin^2\theta\,d\theta=0$, let $\nu=(1+r)\ST$ (in the angle
variable). Since $U_k(2\cos\theta)\sin\theta=\sin((k+1)\theta)$,
\begin{equation}\label{eq:mom}
\widehat\nu(k)=\frac2\pi\int_0^\pi\sin((k+1)\theta)\,r(\theta)\sin\theta\,d\theta\qquad(k\ge1).
\end{equation}

\subsection{Rudin--Shapiro polynomials}
Let $p_0=q_0=1$ and $p_{n+1}(z)=p_n(z)+z^{2^n}q_n(z)$, $q_{n+1}(z)=p_n(z)-z^{2^n}q_n(z)$ \cite{Rudin}. For $N=2^n$ these have degree
$N-1$, coefficients $\pm1$, and $|p_n|^2+|q_n|^2=2N$ on $|z|=1$; in particular $|p_n|,|q_n|\le\sqrt{2N}$ on $|z|\le1$.

Fix $N=2^n$, $M=N$, and let $a_0,\dots,a_{N-1}\in\{\pm1\}$ be the coefficients of $p=p_n$ or of $p=q_n$. Put $m_j=2(M+j)$ and
\[
Q(\theta)=\sum_{j<N}a_j\cos(m_j\theta)=\operatorname{Re}\bigl(z^Mp(z)\bigr),\qquad P=-Q'=\sum_{j<N}m_ja_j\sin(m_j\theta),\qquad z=e^{2i\theta}.
\]
\begin{lemma}\label{lem:RS}
\textup{(i)} $\|P\|_\infty\le4\sqrt2N^{3/2}$.
\textup{(ii)} $R(\theta)=\sum_ja_j\sin(m_j\theta)/m_j$ satisfies $\|R\|_\infty\le(2N)^{-1/2}$.
\textup{(iii)} For one of the two choices $p\in\{p_n,q_n\}$, $\int_0^\pi\sin^4\theta\,|Q(\theta)|\,d\theta\ge\frac{3\pi}{16}\sqrt{N/2}$.
\end{lemma}
\begin{proof}
(i) $\sum_jm_ja_je^{im_j\theta}=z^M(2Mp(z)+2zp'(z))$, and by Bernstein's inequality $|p'|\le(N-1)\sqrt{2N}$ on $|z|=1$; so
$|P|\le(2N+2N)\sqrt{2N}$.
(ii) $\sum_ja_jz^{M+j}/(M+j)=\int_0^1s^{M-1}z^Mp(sz)\,ds$ has modulus at most $\sqrt{2N}/M$, and $R=\frac12\operatorname{Im}$ of this.
(iii) Let $Q_p,Q_q$ be built from $p_n,q_n$. Since $\operatorname{Re}(w)^2=\frac12(|w|^2+\operatorname{Re}w^2)$,
\[
Q_p^2+Q_q^2=\tfrac12\bigl(|p_n|^2+|q_n|^2\bigr)(z)+\tfrac12\operatorname{Re}\bigl(z^{2M}(p_n^2+q_n^2)(z)\bigr)=N+\sum_{l\ge4M}f_l\cos(l\theta)
\]
with real $f_l$, $l$ even. As $\sin^4\theta=\frac38-\frac12\cos2\theta+\frac18\cos4\theta$, $\int_0^\pi\sin^4\theta\cos(l\theta)\,d\theta=0$ for even
$l\ge6$, so $\int_0^\pi\sin^4\theta(Q_p^2+Q_q^2)\,d\theta=\frac{3\pi}8N$. Since $|Q_p|,|Q_q|\le\sqrt{2N}$,
$\int_0^\pi\sin^4\theta(|Q_p|+|Q_q|)\,d\theta\ge\frac1{\sqrt{2N}}\cdot\frac{3\pi}8N=\frac{3\pi}8\sqrt{N/2}$.
\end{proof}

\begin{prop}\label{prop:nuN}
Let $N=2^n\ge2048$, $0<t\le(4\sqrt2N^{3/2})^{-1}$, $p$ as in Lemma~\ref{lem:RS}\textup{(iii)}, $r=t\sin\theta\,P(\theta)$ and
$\nu=(1+r)\ST$. Then $0\le\nu\le2\ST$, $|\widehat\nu(k)|\le(k+1)t$ for $k\ge1$, and $W_1(\nu,\ST)\ge0.50\,t\sqrt N$.
\end{prop}
\begin{proof}
\emph{Positivity.} $|r|\le t\|P\|_\infty\le1$ by Lemma~\ref{lem:RS}(i).

\emph{Moments.} By \eqref{eq:mom}, $\widehat\nu(k)=t\mathfrak d_{k+1}$, where $\mathfrak d_l$ is the $l$-th sine coefficient of $\sin^2\theta\,P(\theta)$. From
$\sin^2\theta\sin m\theta=\frac12\sin m\theta-\frac14\sin(m+2)\theta-\frac14\sin(m-2)\theta$ (all $m_j\ge4$), with $e_m=a_j$ if $m=m_j$ and
$e_m=0$ otherwise, $\mathfrak d_l=\frac12le_l-\frac14(l-2)e_{l-2}-\frac14(l+2)e_{l+2}$, so $|\mathfrak d_l|\le l$. Only even $l$ occur, so $\mathfrak d_1=0$, i.e.\
$\int r\sin^2\theta\,d\theta=0$.

\emph{Lower bound.} Let $H(x)=\nu((-\infty,x])-G(x)$. Since $\{x'\le x\}=\{\theta'\ge\theta\}$,
$H(2\cos\theta)=-\frac2\pi\int_0^\theta r(u)\sin^2u\,du$ and $W_1(\nu,\ST)=\int_0^\pi|H(2\cos\theta)|\,2\sin\theta\,d\theta$. Put
$I(\theta)=\int_0^\theta\sin^3u\,P(u)\,du$, so $|H(2\cos\theta)|=\frac{2t}\pi|I(\theta)|$. Integrating by parts twice, with
$w(u)=\sin^2u\cos u$ and $R'=Q$,
\[
I(\theta)=-\sin^3\theta\,Q(\theta)+3J(\theta),\qquad J(\theta)=\int_0^\theta wQ\,du=w(\theta)R(\theta)-\int_0^\theta w'R\,du .
\]
Since $|w|\le1$ and $|w'|\le3$, Lemma~\ref{lem:RS}(ii) gives $|J|\le(1+3\pi)(2N)^{-1/2}\le7.5N^{-1/2}$. Hence, by Lemma~\ref{lem:RS}(iii),
\[
W_1(\nu,\ST)\ge\frac{4t}\pi\Bigl(\int_0^\pi\sin^4\theta\,|Q|\,d\theta-6\|J\|_\infty\Bigr)\ge\frac{4t}\pi\Bigl(\frac{3\pi}{16\sqrt2}-\frac{45}N\Bigr)\sqrt N\ge0.50\,t\sqrt N,
\]
using $N\ge2048$.
\end{proof}

\subsection{A bump in the bulk}
\begin{prop}\label{prop:bulk}
Let $0<\varepsilon\le\frac1{10}$, $c=0.3$, and let $\nu$ have density $\varrho(x)+c\phi'(x/\varepsilon)$, where $\phi(u)=\max(0,1-|u|)$. Then
$\frac1{20}\ST\le\nu\le2\ST$, $|\widehat\nu(k)|\le0.155(k+1)\varepsilon^2$ for $k\ge1$, $D(\nu)\ge c\varepsilon$, $W_\infty(\nu,\ST)\ge\pi c\varepsilon$ and
$W_s(\nu,\ST)\ge0.159^{1/s}\cdot0.0375\,\varepsilon^{1+1/s}$ for $1\le s<\infty$.
\end{prop}
\begin{proof}
For $|x|\le\frac1{10}$, $\varrho(x)\ge\varrho(\frac1{10})>0.3179$, and $|c\phi'|\le0.3$; this gives the comparison with $\ST$, and
$\int\phi'=0$. Since $\int U_k(x)\phi'(x/\varepsilon)\,dx=-\varepsilon^2\int U_k'(\varepsilon u)\phi(u)\,du$ and $\int\phi=1$, it suffices to bound $|U_k'|$
for $|x|\le\frac1{10}$, i.e.\ $|\cos\theta|\le\frac1{20}$. From
$\frac d{d\theta}U_k(2\cos\theta)=\frac{(k+1)\cos((k+1)\theta)\sin\theta-\sin((k+1)\theta)\cos\theta}{\sin^2\theta}$ and $dx=-2\sin\theta\,d\theta$,
\[
|U_k'(x)|\le\frac{k+1}{2\sin^2\theta}+\frac{|\cos\theta|}{2\sin^3\theta}\le0.5013(k+1)+0.0251\le0.514(k+1),
\]
so $|\widehat\nu(k)|\le0.3\cdot0.514(k+1)\varepsilon^2<0.155(k+1)\varepsilon^2$. Next, $F_\nu(x)-G(x)=c\varepsilon\phi(x/\varepsilon)$, which equals
$c\varepsilon$ at $x=0$. Next, for $t=F_\nu(0)=\frac12+c\varepsilon$ we have $q_\nu(t)\le0$, while $\mathfrak q(t)\ge(t-\frac12)/\max\varrho=\pi c\varepsilon$; since $q_\nu$ and $\mathfrak q$ are continuous, this gives $W_\infty(\nu,\ST)\ge\pi c\varepsilon$ by \eqref{eq:quantile}.
Moreover, for $|y|\le\varepsilon/4$, $t=G(y)$ and $w=c\varepsilon/8$ we have
$F_\nu(y-w)\ge G(y)-\frac w\pi+c\varepsilon(1-\frac14-\frac c8)>t$, so $q_\nu(t)\le y-w$; these $t$ form a set of measure
$\ge\frac\varepsilon2\varrho(\frac1{10})>0.159\varepsilon$, so $W_s(\nu,\ST)\ge0.159^{1/s}\cdot0.0375\,\varepsilon^{1+1/s}$.
\end{proof}

\subsection{A hole at the edge}
\begin{prop}\label{prop:edge}
Let $0<\ell\le\frac12$, $J=[\frac\pi2-\frac1{10},\frac\pi2+\frac1{10}]$, $a=\int_0^\ell\sin^2/\int_J\sin^2$, and $r=-1$ on $[0,\ell]$, $r=a$ on
$J$, $r=0$ elsewhere. Then $\nu=(1+r)\ST$ satisfies $\nu\le2\ST$, $|\widehat\nu(k)|\le0.319(k+1)\ell^3$ for $k\ge1$, and
$\nu^\theta([0,\ell))=0$.
\end{prop}
\begin{proof}
Here $a\le\frac{\ell^3/3}{0.2\cos^2(0.1)}<\frac14$ and $\int r\sin^2=0$. In \eqref{eq:mom}, $|\sin((k+1)\theta)|\le(k+1)\sin\theta$ bounds the
contribution of $[0,\ell]$ by $\frac2\pi(k+1)\int_0^\ell\sin^2\le\frac2{3\pi}(k+1)\ell^3$, and $|\sin((k+1)\theta)|\le1\le\frac{\sin\theta}{\cos(0.1)}$
on $J$ bounds the contribution of $J$ by $\frac2\pi\cdot\frac{a}{\cos(0.1)}\int_J\sin^2\le\frac2{3\pi\cos(0.1)}\ell^3$. Since $k+1\ge2$, the sum is
at most $\frac2{3\pi}\bigl(1+\frac1{2\cos(0.1)}\bigr)(k+1)\ell^3<0.319(k+1)\ell^3$.
\end{proof}

\begin{proof}[Proof of Theorem~\ref{thm:D}]
Let $0<\delta\le10^{-6}$ and $\delta'=\frac9{10}\delta$. In each case we build $\nu$ with $\nu\le2\ST$ and $|\widehat\nu(k)|\le(k+1)\delta'$, and let
$\mu_i$ be its discretization from Lemma~\ref{lem:discr}, so $|\widehat{\mu_i}(k)|\le(k+1)\delta$ and $\frac2L\le\frac{\delta^2}{21}$.

\emph{$\mu_1$.} Let $X=(4\sqrt2\delta')^{-2/3}\ge3379$ and let $N$ be the largest power of $2$ with $N\le X$, so $N\ge2048$ and $N>X/2$.
Proposition~\ref{prop:nuN} with $t=\delta'$ gives
$W_1(\nu,\ST)\ge0.50\,\delta'(X/2)^{1/2}=0.50\cdot2^{-1/2}(4\sqrt2)^{-1/3}\delta'^{2/3}>0.198\,\delta'^{2/3}>0.18\,\delta^{2/3}$, so
$W_1(\mu_1,\ST)\ge0.18\delta^{2/3}-\frac{\delta^2}{21}\ge\frac16\delta^{2/3}$.

\emph{$\mu_2$.} Take $\varepsilon=(\delta'/0.155)^{1/2}\le\frac1{10}$ in Proposition~\ref{prop:bulk}. Then
$D(\nu)\ge0.3\varepsilon=0.3(0.9/0.155)^{1/2}\delta^{1/2}>0.72\delta^{1/2}$ and $W_\infty(\nu,\ST)\ge0.3\pi\varepsilon>2.27\delta^{1/2}$. Hence
$D(\mu_2)\ge0.72\delta^{1/2}-\frac1{2L}\ge0.7\delta^{1/2}$. Since $\nu\ge\frac1{20}\ST$, each interval $[q_\nu(\frac{i-1}L),q_\nu(\frac iL)]$ has
$\ST$-mass at most $\frac{20}L<\frac12$, hence length $<2$; an interval of length $\lambda\le2$ has $\ST$-mass at least
$\frac{\sqrt2}{3\pi}\lambda^{3/2}$ (Lemma~\ref{lem:q}), so these intervals have length at most $(\frac{60\pi}{\sqrt2L})^{2/3}\le2.2\delta^{4/3}$. Hence
$W_\infty(\mu_2,\ST)\ge2.27\delta^{1/2}-2.2\delta^{4/3}\ge2.2\delta^{1/2}$. Similarly, for $s\ge2$,
$W_s(\mu_2,\ST)\ge W_s(\nu,\ST)-W_\infty(\mu_2,\nu)\ge0.159\cdot0.0375\cdot2.40\,\delta^{\frac12+\frac1{2s}}-2.2\delta^{4/3}\ge\frac1{100}\delta^{\frac12+\frac1{2s}}$,
since $\varepsilon\ge2.40\delta^{1/2}$, $\varepsilon^{1+1/s}\ge2.40\delta^{\frac12+\frac1{2s}}$ and $\frac12+\frac1{2s}\le\frac34$.

\emph{$\mu_3$.} Take $\ell=(\delta'/0.319)^{1/3}\le\frac12$ in Proposition~\ref{prop:edge}. The atoms of $\mu_3$ lie in the support of $\nu$, so
$\mu_3^\theta$ has no mass on $[0,\ell)$. Hence $q_{\mu_3^\theta}(t)\ge\ell$ for all $t>0$, while $q_{\ST^\theta}(t)\to0$ as $t\to0$, so
$W_\infty(\mu_3^\theta,\ST^\theta)\ge\ell=(0.9/0.319)^{1/3}\delta^{1/3}>1.4\delta^{1/3}$. Moreover, for $t\le\Psi(\ell/2)$ we have
$q_{\mu_3^\theta}(t)-q_{\ST^\theta}(t)\ge\ell/2$, and $\Psi(\ell/2)\ge\frac15(\ell/2)^3$ (as $\ell/2\le\frac14$), so
$W_s(\mu_3^\theta,\ST^\theta)\ge(\frac15)^{1/s}(\ell/2)^{1+3/s}\ge\frac15(0.7065)^4\delta^{\frac13+\frac1s}>\frac1{25}\delta^{\frac13+\frac1s}$.
\end{proof}

\begin{remark}
The measures in Theorem~\ref{thm:D} are empirical measures of multisets: points may be repeated, as for $\mu_q$ itself. In the
choice $L=q-1$, $\delta=7q^{-1/2}$ made after Theorem~\ref{thm:D}, the moment constant $7$ is about $14$ times Katz's constant
$|\widehat{\mu_q}(k)|\le\frac{(k+1)\sqrt q}{2(q-1)}$; this comes from the constant $42$ in Lemma~\ref{lem:discr} and does not affect exponents.
\end{remark}

\begin{remark}\label{rem:gap}
Since $W_s(\mu^\theta,\ST^\theta)\ge\frac12W_s(\mu,\ST)$, the measures $\mu_1$ and $\mu_2$ also give the lower bounds $\delta^{2/3}$ and
$\delta^{\frac12+\frac1{2s}}$ ($s\ge2$) in the angle metric. With $\mu_3$, these match Theorem~\ref{thm:C} for $s\le2$, $s>6$ and $s=\infty$;
for $2<s<6$ there remains a gap (Question~\ref{q:gap}).
\end{remark}

\section{Proofs of Theorems~\ref{thm:C} and \ref{thm:A}}\label{sec:proofs}

\begin{proof}[Proof of Theorem~\ref{thm:C}]
\emph{Kolmogorov distance.} $D(\mu)\ll\eta^{1/2}$ is part of Theorem~\ref{thm:Winf}. If $\mu$ is the empirical measure of a finite
multiset, use Proposition~\ref{prop:Nied}: if $\eta>\frac19$ then $D(\mu)\le1<3\eta^{1/2}$. Otherwise let $x=\eta^{-1/2}\ge3$ and let $k$ be an
odd integer in $[x,x+2)$. In Proposition~\ref{prop:Nied} the first term is $<7.04/k$; since
$\sum_{r=1}^{2k-1}\frac{(r+1)^2}{r(r+2)}<2k$ and $\frac{4k-3}{0.362\pi k+2\pi}<3.52$, the second is $<7.04k\eta$. As $1/k\le\eta^{1/2}$ and
$k\eta<\eta^{1/2}+2\eta\le\frac53\eta^{1/2}$, we get $D(\mu)<7.04\cdot\frac83\eta^{1/2}<19\eta^{1/2}$.

\emph{$W_s$, $1\le s<\infty$}, in the trace metric is Theorem~\ref{thm:sharp}; the angle bounds were proved after
Theorem~\ref{thm:sharp}, and \emph{$W_\infty$} is Theorem~\ref{thm:Winf}.

\emph{Explicit $W_2$.} For $\eta>1$ use $W_2\le\pi$. Otherwise take $T=\lceil\eta^{-2/3}\rceil\le2\eta^{-2/3}$ in Theorem~\ref{thm:W2}. Since $(k+1)^2/(k(k+2))\le\frac43$,
$2(\sum_{k\le T}|\widehat\mu(k)|^2/(k(k+2)))^{1/2}\le\frac4{\sqrt3}\eta\sqrt T\le\frac{4\sqrt2}{\sqrt3}\eta^{2/3}<3.3\,\eta^{2/3}$, and $\pi^2/T\le\pi^2\eta^{2/3}$.
\end{proof}

\begin{lemma}\label{lem:third}
Let $\F_q$ have characteristic $\ne3$, and let $N_3\in\{0,2\}$ be the number of roots of $t^2+t+1$ in $\F_q$. Then, with
$K(a)=\sqrt q\,\Kl(a;q)$,
\[
\sum_{a\in\F_q^\times}K(a)=1,\qquad\sum_{a\in\F_q^\times}K(a)^3=(N_3-1)q^2+2q+1 .
\]
\end{lemma}
For $q=p$ this is classical, $\sum_aK(a)^3=(\frac p3)p^2+2p+1$ \cite{Salie,XY}; we include the short proof for general $\F_q$.
\begin{proof}
The first identity follows from $\sum_{a\ne0}\psi(a/x)=-1$. For the second, expanding the cube and summing over $a$,
\[
\sum_aK(a)^3=\sum_{x,y,z\in\F_q^\times}\psi(x+y+z)\bigl(q\,\mathbf 1[x^{-1}+y^{-1}+z^{-1}=0]-1\bigr)=qT+1,
\]
where $T$ is the sum of $\psi(x+y+z)$ over $x^{-1}+y^{-1}+z^{-1}=0$, because $(\sum_{x\ne0}\psi(x))^3=-1$. With $a=x^{-1}$, $b=y^{-1}$,
$c=z^{-1}=-a-b$, one has $x+y+z=\frac{ab+bc+ca}{abc}=\frac{a^2+ab+b^2}{ab(a+b)}$. Writing $a=\lambda$, $b=t\lambda$ with $\lambda\ne0$ and
$t\notin\{0,-1\}$, this equals $m(t)/\lambda$ with $m(t)=\frac{1+t+t^2}{t(1+t)}$. Summing over $\lambda$ gives $q-1$ if $m(t)=0$ and $-1$
otherwise, and $m(t)=0$ for exactly $N_3$ values of $t$. Hence $T=N_3(q-1)-(q-2-N_3)=N_3q-q+2$.
\end{proof}

\begin{proof}[Proof of Theorem~\ref{thm:A}]
(1) By Proposition~\ref{prop:Katz}, Theorem~\ref{thm:C} applies to $\mu_q$ with $\eta=q^{-1/2}$, and
$\ell_\eta\le\log q$ (for $q\ge4$; small $q$ are trivial). The explicit $W_2$ bound is Corollary~\ref{cor:W2Kl}.

(2) For $s\le2$ use Corollary~\ref{cor:W2Kl}, and for $s>2$ the angle bounds of Theorem~\ref{thm:C} with the same $\eta$ as in (1).

(3) Let $\eta=q^{-1/2}$ and $\tau=A\eta^{1/3}$, and assume $\tau\le\frac12$ (otherwise take $C$ large and $c$ small). Since
$\tau-\sin\tau\cos\tau\ge\frac23\tau^3(1-\frac{\tau^2}5)$, $\Psi(\tau)\ge\frac15\tau^3=\frac15A^3\eta$. By Proposition~\ref{prop:local},
$\mu_q^\theta([0,\tau])\ge\frac15A^3\eta-C_1(1+A^{3/2})\eta\ge\frac1{10}A^3\eta$ for $A$ large. Thus at least
$\frac1{10}A^3(q-1)q^{-1/2}$ elements $a$ have $\theta_a\le\tau$, and for them $\Kl(a;q)=2\cos\theta_a\ge2-\theta_a^2\ge2-A^2q^{-1/3}$.
Applying this to $\mu_q^-$ gives the statement at $-2$.

(4) Let $3\nmid q$. Then $N_3-1=\pm1$ in
Lemma~\ref{lem:third}, so $\sum_aU_3(\Kl(a;q))=\sum_a(\Kl^3-2\Kl)=q^{-3/2}(\pm q^2+2q+1)-2q^{-1/2}=\pm q^{1/2}+q^{-3/2}$. Hence
$|\widehat{\mu_q}(3)|\ge(q^{1/2}-q^{-3/2})/(q-1)=(q+1)q^{-3/2}>q^{-1/2}$. As $\int U_3\,d\ST=0$ and $U_3(x)=x^3-2x$ is $10$-Lipschitz on
$[-2,2]$ with total variation $<12.4$, we get $W_1(\mu_q,\ST)>\frac1{10}q^{-1/2}$ and $D(\mu_q)>\frac1{13}q^{-1/2}$. Finally
$W_1(\mu_q^\theta,\ST^\theta)\ge\frac12W_1(\mu_q,\ST)$.
\end{proof}

\begin{remark}
The exponent $\frac13$ in (3) is the limit of what Katz's bounds give: $\mu_3$ in Theorem~\ref{thm:D} satisfies
$|\widehat{\mu_3}(k)|\le(k+1)\delta$ and has no mass on $(2\cos\ell,2]$, where $2-2\cos\ell\ge\ell^2(1-\ell^2/12)>1.9\,\delta^{2/3}$. Numerically,
$2-\max_a\Kl(a;p)$ is of order $p^{-2/3}$ for $10^4\le p\le4\cdot10^6$.
\end{remark}

\begin{remark}\label{rem:BC}
Borda and Cuenin \cite[Thm.~4]{BC} bound $W_s$ on a compact Lie group in terms of the Fourier coefficients of the measures.
On $\mathrm{SU}_2$, applied to Haar measure and the central lift $\tilde\mu_q$, with $|\widehat{\mu_q}(k)|\le(k+1)q^{-1/2}$, it gives
$W_s(\mu_q^\theta,\ST^\theta)\ll q^{-1/3}$ for $s\le2$ and $\ll q^{-r/6}$ with $r=s/(s-1)$ for $s>2$. For $s>2$ the exponent
$\frac s{6(s-1)}$ is smaller than those in Theorem~\ref{thm:A}(2) (for $2<s\le6$, $\frac5{24}+\frac1{4s}-\frac s{6(s-1)}=\frac{(s-2)(s+3)}{24s(s-1)}>0$, and for $s>6$,
$\frac16+\frac1{2s}-\frac s{6(s-1)}=\frac{2s-3}{6s(s-1)}>0$), and also smaller than those of Theorem~\ref{thm:A}(1) in the trace metric.
\end{remark}

\begin{remark}\label{rem:numerics}
We computed $\Kl(a;p)$ for all $a$ exactly, as a multiplicative convolution evaluated by the fast Fourier transform on $\F_p^\times$
(i.e.\ via Gauss sums), and $W_s$ from \eqref{eq:quantile} with the exact empirical quantiles and $\mathfrak q$, $q_{\ST^\theta}$ obtained by
inverting the closed forms of $G$ and $\Psi$ numerically (error $<10^{-8}$, negligible against the distances measured). For
$p=1009$, $10007$, $100003$, $1000003$, $4000037$ one finds
$W_1(\mu_p,\ST)\sqrt p\in[0.47,0.67]$, $W_2(\mu_p,\ST)\sqrt p\in[0.59,0.83]$, $D(\mu_p)\sqrt p\in[0.43,0.67]$, and
$W_\infty(\mu_p^\theta,\ST^\theta)p^{1/3}\in[0.9,2.4]$; the first three are consistent with Theorem~\ref{thm:A}(4).
\end{remark}

\subsection*{Families with \texorpdfstring{$\mathrm{SL}_2$}{SL2} monodromy}

\begin{cor}\label{cor:SL2}
Let $U\subset\mathbf P^1_{\F_q}$ be a dense open set, $S=(\mathbf P^1\smallsetminus U)(\overline{\F}_q)$, and $\mathcal F$ a lisse
$\overline{\mathbf Q}_\ell$-sheaf of rank $2$ on $U$, pointwise pure of weight $0$, with geometric and arithmetic monodromy groups
$\mathrm{SL}_2$. Put $c=|S|+\sum_{x\in S}\Sw_x(\mathcal F)$ and $\mu_{\mathcal F}=|U(\F_q)|^{-1}\sum_{x\in U(\F_q)}\delta_{\tr(\mathrm{Frob}_x\mid\mathcal F)}$. If
$q\ge2|S|$, then $|\widehat{\mu_{\mathcal F}}(k)|\le2c(k+1)q^{-1/2}$ for $k\ge1$, so Theorem~\ref{thm:C} applies with $\eta=2cq^{-1/2}$:
\[
W_s(\mu_{\mathcal F},\ST)\ll_s\begin{cases}c^{2/3}q^{-1/3},&1\le s<3,\\ c^{2/3}q^{-1/3}(\log q)^{1/3},&s=3,\\ c^{\frac12+\frac1{2s}}q^{-\frac14-\frac1{4s}},&3<s<\infty,\\ c^{1/2}q^{-1/4},&s=\infty,\end{cases}
\qquad W_\infty(\mu_{\mathcal F}^\theta,\ST^\theta)\ll c^{1/3}q^{-1/6}.
\]
\end{cor}
\begin{proof}
For $k\ge1$, $\Sym^k\mathcal F$ is geometrically irreducible and nontrivial, so $H^0_c=H^2_c=0$ and
$\sum_{x\in U(\F_q)}U_k(\tr\mathrm{Frob}_x)=-\tr(\mathrm{Frob}\mid H^1_c(U_{\overline{\F}_q},\Sym^k\mathcal F))$. By Deligne the eigenvalues have
modulus $\le q^{1/2}$. By the Grothendieck--Ogg--Shafarevich formula, $\dim H^1_c=(k+1)(|S|-2)+\sum_{x\in S}\Sw_x(\Sym^k\mathcal F)$, and
$\Sw_x(\Sym^k\mathcal F)\le(k+1)\Sw_x(\mathcal F)$ since $\Sym^k\mathcal F$ is a direct summand of $\mathcal F^{\otimes k}$, whose breaks at $x$ are at most
the largest break of $\mathcal F$ \cite[1.3]{Katz}, and this break is at most $\Sw_x(\mathcal F)$, the sum of the two breaks. Finally $|U(\F_q)|\ge q+1-|S|\ge q/2$.
\end{proof}

Examples: Kloosterman sums ($\mathcal K\ell_2$ on $\mathbf G_m$, $c=3$, \cite[Thm.~4.1.1]{Katz}); the Legendre family $y^2=x(x-1)(x-\lambda)$ for $q$ odd, with
$\mathcal F=R^1\pi_*\overline{\mathbf Q}_\ell(\frac12)$ on $\mathbf P^1\smallsetminus\{0,1,\infty\}$ for a choice of $\sqrt q$ ($c=3$), for which the
Kolmogorov bound $q^{-1/4}$ is due to Michel \cite{Michel} and Miller--Murty \cite{MM}; and the Birch sums $\sum_x\psi(ax+x^3)$ over
$\F_p$, $p>7$ \cite{KatzBirch} ($c=4$).

\section{Beyond Katz's bounds: a random model (heuristic)}\label{sec:model}

Apart from Proposition~\ref{prop:conditional}, this section is heuristic and is not used elsewhere. By Theorem~\ref{thm:D}, improving
the exponent $\frac13$ for $W_2$ requires information beyond Katz's bound, which treats the roughly $k/2$ Frobenius eigenvalues on
$H^1_c(\mathbf G_{m,\overline{\F}_p},\Sym^k\mathcal K\ell_2)$ trivially. Put $z_k=\sqrt p\,\widehat{\mu_p}(k)$; Katz's bound is $|z_k|\le k+1$. If the
angles $\theta_a$ behaved like independent samples from $\ST^\theta$, the $z_k$ would be approximately standard Gaussian. Numerically,
over the $600$ primes in $[30011,36229]$ and for each of the ranges $9\le k\le30$, $31\le k\le100$, \dots, $301\le k\le400$, the mean of
$z_k^2$ lies in $[0.986,0.998]$ and that of $z_k^4$ in $[2.94,3.03]$; for $p=100003$ and $p=1000003$ the $z_k$ with $9\le k\le2\sqrt p$ satisfy
$|\mathrm{corr}(z_k,z_{k+l})|\le0.07$ for $1\le l\le4$ (correlation over $k$) and $\max_k|z_k|\le3.15<(2\log(2\sqrt p))^{1/2}$.

\begin{prop}[conditional]\label{prop:conditional}
Suppose that $|\widehat{\mu_p}(k)|\le B\,p^{-1/2}(\log p)^{1/2}$ for $1\le k\le p^{1/2}$. Then
$W_2(\mu_p^\theta,\ST^\theta)\le\pi^2p^{-1/2}(1+o(1))+\sqrt3B\,p^{-1/2}(\log p)^{1/2}$ and $D(\mu_p)\ll B\,p^{-1/2}(\log p)^{3/2}$. If instead only the
averaged bound
\begin{equation}\label{eq:hypB}
\frac1{\#\{X<p\le2X\}}\sum_{X<p\le2X}p\,|\widehat{\mu_p}(k)|^2\le B\qquad(1\le k\le X^{1/2})
\end{equation}
holds, then the mean of $W_2(\mu_p^\theta,\ST^\theta)^2$ over $X<p\le2X$ is $\ll(1+B)/X$, so $W_2\ll p^{-1/2+\varepsilon}$ for almost all $p$.
\end{prop}
\begin{proof}
Take $T=\lfloor p^{1/2}\rfloor$ (resp.\ $\lfloor X^{1/2}\rfloor$) in the angle-metric inequality of Theorem~\ref{thm:W2}, whose first term
is $\pi^2/T=\pi^2p^{-1/2}(1+o(1))$, and use $\sum_k1/(k(k+2))=\frac34$; in the averaged case, square and average. For $D$, take an odd
$k$ with $k\le(p^{1/2}+1)/2$ and $k\asymp p^{1/2}$ in Proposition~\ref{prop:Nied}; then $2k-1\le p^{1/2}$, and the second term is
$\ll\sum_{r<2k}r^{-1}|\widehat{\mu_p}(r)|\ll Bp^{-1/2}(\log p)^{3/2}$.
\end{proof}

\begin{conj}\label{conj:gauss}
There is an absolute $B$ such that $|\widehat{\mu_p}(k)|\le Bp^{-1/2}(\log p)^{1/2}$ for all primes $p$ and $1\le k\le p^{1/2}$.
\end{conj}

For fixed $k$, $z_k$ is essentially a normalized Frobenius trace on the pure part of the cohomology group above, which comes
from a motive over $\mathbf Q$ \cite{FSY}. As $p$ varies, a conjecture of Sato--Tate type predicts that its limiting distribution is
the law of the trace of a Haar-random element of the Sato--Tate group of this motive; in particular it is bounded. When this group is a large orthogonal or symplectic group, the low moments of this trace agree with those of
a standard Gaussian \cite{DS}, consistent with the data above, and one expects Gaussian behaviour only as $k\to\infty$. For
$5\le k\le8$ these motives are modular ($k=5$ by Livn\'e \cite{Livne} and Peters, Top and van der Vlugt \cite{PTV}, $k=6$ by Hulek, Spandaw,
van Geemen and van Straten \cite{HSGS}, $k=7,8$ by Yun \cite{Yun}). For fixed $k$, Katz's bound gives \eqref{eq:hypB} with
$B=(k+1)^2$; the content of \eqref{eq:hypB} is the uniformity in $k$, and no uniform result is known for $1\le k\le X^\gamma$ with any
$\gamma>0$.

\section{Questions}\label{sec:questions}

\begin{enumerate}
\item Is $W_2(\mu_p,\ST)\ll p^{-\gamma}$ for some $\gamma>\frac13$, for all or for almost all $p$? By Theorem~\ref{thm:D} this requires
arithmetic information beyond Katz's bounds; by Proposition~\ref{prop:conditional} the averaged hypothesis \eqref{eq:hypB} would suffice for
almost all $p$, and Conjecture~\ref{conj:gauss} would give $W_2\ll p^{-1/2}(\log p)^{1/2}$ for all $p$.
\item In the trace metric, is $W_s(\mu_p,\ST)\ll_{s,\varepsilon}p^{-1/2+\varepsilon}$ for every $1\le s\le\infty$? By Theorem~\ref{thm:A}(4) this
would be optimal up to $p^\varepsilon$. In the angle metric, what is the true order of $W_\infty(\mu_p^\theta,\ST^\theta)$, which lies between
$p^{-1/3}$ and $p^{-1/6}$? (The lower bound $W_\infty\ge\frac12\Psi^{-1}(\frac1{p-1})>\frac45p^{-1/3}$ follows from \eqref{eq:quantile},
since $q_{\mu_p^\theta}$ is constant on $(0,\frac1{p-1}]$ and $\Psi(\tau)\le\frac2{3\pi}\tau^3$.)
\item\label{q:gap} Is the logarithm at $s=3$ in Theorem~\ref{thm:C} necessary? In the angle metric, for $2<s<6$, the optimal exponent $a$ of
$\eta$ (up to logarithms) satisfies $\frac5{12}+\frac1{2s}\le a\le\min(\frac23,\frac13+\frac1s)$ by Theorems~\ref{thm:C} and \ref{thm:D}; what is it?
\item What are the analogues of Theorems~\ref{thm:A}--\ref{thm:D} for hyper-Kloosterman sums $\Kl_n$, $n\ge3$, whose limiting laws
are trace laws of $\mathrm{USp}(n)$ or $\mathrm{SU}(n)$? For $n$ odd the traces are complex and there is no quantile coupling.
\end{enumerate}

\end{document}